\documentclass[10pt, a4paper, oneside, reqno]{amsart}
\usepackage[numbered]{bookmark}
\usepackage{changepage} 
\usepackage[utf8]{inputenc}
\usepackage{graphics, graphicx}
\usepackage{amsmath, amsfonts, amsthm, amssymb}
\usepackage{mathrsfs, mathtools, mathabx}
\usepackage{verbatim} 
\usepackage{enumerate}
\usepackage[usenames,dvipsnames]{xcolor}
\usepackage{cite}
\usepackage[T1]{fontenc}
\usepackage[normalem]{ulem}

\newcommand{\Z}{\mathbb{Z}} 
\newcommand{\N}{\mathbb{N}}
\newcommand{\Q}{\mathbb{Q}}
\newcommand{\R}{\mathbb{R}}

\newcommand{\A}{\mathcal{A}}

\newcommand{\Function}[5]{\begin{array}{cccc} #1 : & #2 & \rightarrow & #3 \\ & #4 & \mapsto & #5 \end{array}}

\newcommand\beq[1]{\begin{equation}\label{#1}}
\newcommand{\eeq}{\end{equation}}

\newcommand{\beqa}{ \begin{align*} }
\newcommand{\eeqa}{ \end{align*} }

\newcommand{\beqano}{ \begin{eqnarray*} }
\newcommand{\eeqano}{ \end{eqnarray*} }
\newtheorem{theorem}{Theorem}[section]
\newtheorem{definition}{Definition}[section]
\newtheorem{proposition}{Proposition}[section]
\newtheorem{lemma}{Lemma}[section]
\newtheorem{sublemma}{Sublemma}[section]
\newtheorem{remark}{Remark}[section]
\newtheorem{notationalremark}{Notations}[section]
\newtheorem{corollary}{Corollary}[section]
\newtheorem{assumption}{Assumption}[section]
\newtheorem{claim}{Claim}[section]

\newcommand\thm[1]{ \begin{theorem}\label{#1}}
\newcommand\thmtwo[2]{ \begin{theorem}[#1]\label{#2}}
\newcommand\ethm{ \end{theorem} }
\newcommand\dfn[1]{ \begin{definition}\label{#1} \rm}
\newcommand\dfntwo[2]{ \begin{definition}[#1]\label{#2} \rm}
\newcommand\edfn{ \end{definition} }
\newcommand\pro[1]{ \begin{proposition}\label{#1}}
\newcommand\protwo[2]{ \begin{proposition}[#1]\label{#2}}
\newcommand\epro{ \end{proposition} }
\newcommand\lem[1]{ \begin{lemma}\label{#1}}
\newcommand\lemtwo[2]{ \begin{lemma}[#1]\label{#2}}
\newcommand\elem{ \end{lemma} }
\newcommand\sublem[1]{ \begin{sublemma}\label{#1}}
\newcommand\sublemtwo[2]{ \begin{sublemma}[#1]\label{#2}}
\newcommand\esublem{ \end{sublemma} }
\newcommand\rem[1]{ \begin{remark}\label{#1} \rm}
\newcommand\erem{ \end{remark} }
\newcommand\notrem[1]{ \begin{notationalremark}\label{#1} \rm}
\newcommand\enotrem{ \end{notationalremark} }
\newcommand\cor[1]{ \begin{corollary}\label{#1}}
\newcommand\cortwo[2]{ \begin{corollary}[#1]\label{#2}}
\newcommand\ecor{ \end{corollary} }
\newcommand\asmp[1]{ \begin{assumption}\label{#1}}
\newcommand\asmptwo[2]{ \begin{assumption}[#1]\label{#2}}
\newcommand\easmp{ \end{assumption} }
\newcommand\clm[1]{ \begin{claim}\label{#1}}
\newcommand\eclm{ \end{claim} }

\newcommand\su[1]{ \frac{1}{ {#1}} }

\newcommand{\e }{ {\epsilon} }

\newcommand{\Leb}{\textup{Leb}_\R}
\newcommand{\LebIR}{\textup{Leb}_{[0, 1) \times \R}}

\newcommand{\RVExt}{\mathcal{R}}
\newcommand{\RVKZ}{\mathcal{Z}}

\begin{document}
\date{}
\title[Ergodicity of skew products over typical IETs]{Ergodicity of skew products over typical IETs}
\author
{ F. Argentieri, P. Berk, F. Trujillo}
\maketitle

\begin{abstract}
We prove the ergodicity of a class of infinite measure-preserving systems of the form
 \[
 \Function{T_f}{[0, 1) \times \R}{[0, 1) \times \R}{(x, t)}{(T(x), t+f(x))},
 \]
where $T$ is an interval exchange transformation and $f$ is a piecewise constant function with a finite number of discontinuities. We show that such a system is ergodic with respect to $\textup{Leb}_{[0, 1)\times \R}$ for a typical choice of parameters of $T$ and $f$. This generalizes a recent result by Chaika and Robertson \cite{chaika_ergodicity_2019} concerning an exceptional class of interval exchange transformations.
\end{abstract}

\section{Introduction}

\emph{Interval exchange transformations}, or IETs, for short (see Section \ref{sc:IETs} for a precise definition), appear naturally as the Poincaré return maps to a transversal section for many classical surface flows, such as translation flows or locally Hamiltonian flows. As such, during the last century, they attracted a lot of attention from a large number of mathematicians. In this article, we study infinite measure-preserving systems called \emph{skew products}, built as extensions of these transformations. We focus on the ergodic properties of the aforementioned systems.

To be more precise, given $m \geq 1$ and $M > 0$, define 
\[ P_{m, M} = \left.\left\{ (p, q) \in \mathbb{R}^{m + 1} \times \mathbb{R}^{m + 1} \,\right|\, \langle p, q \rangle = 0;\, |p|_1 = 1;\, |q|_\infty \leq M\right\}.\] 

We denote by $C_{m, M}$ the space of mean-zero step functions $f:[0, 1) \to \R$ of the form
\[f= q_1\chi_{[0,p_1)}+...+q_{m + 1}\chi_{[p_1+\dots +p_{m+1},1)},\]
for some $(p, q) \in P_{m, M}$. For $f$ as above, we call 
\begin{equation}
\label{eq:jumps}
\sigma_i(f):= q_{i + 1} - q_i
\end{equation}
the \emph{$i$-th jump} of $f$, for any $1 \leq i \leq m$, and denote the set of all jumps of $f$ by $\sigma(f)$. 

We endow $C_{m, M}$ with the Lebesgue measure and the $l_{\infty}$ metric inherited from $P_{m, M}$. Given $f\in C_{m, M}$ and an IET $T$ on $[0, 1)$, we associate the \emph{skew product} $T_f$ given by
\begin{equation}
\label{eq:skew_product}
\Function{T_f}{[0, 1) \times \R}{[0, 1) \times \R}{(x, t)}{(T(x), t+f(x))}.
\end{equation}
We refer to $T$ as the \emph{base} of $T_f$ and to $f$ as the \emph{cocycle} associated with $T_f$. The iterates of $T_f$ are given by
\[T_f^n(x, t) = (T^n(x), t + S_nf(x)),\]
for any $n \in \Z$, where $S_nf$ denotes the associated $n$-th Birkhoff sum. For every $n\in\Z$, the $n$-th \emph{Birkhoff sum} $S_n f: [0, 1) \to \R$ of the cocycle $f$ with respect to the base transformation $T$ is given by 
\begin{equation}
	\label{eq:Birkhoff_sums}
	S_n f(x):=\begin{cases}
		\sum_{i=0}^{n-1}f(T^{i}(x))& \text{ if }n\ge 1,\\
		0&\text{ if }n=0,\\
		-\sum_{i=-n}^{-1}f(T^{i}(x))&\text{ if }n\le 1,
	\end{cases}
\end{equation}
where, for the sake of simplicity, we omit the dependence on $T$ in the notation. 

Skew products of the form \eqref{eq:skew_product} may be seen as first return maps of certain extensions of non-minimal locally Hamiltonian flows (skew product flows with respect to an appropriate cocycle) for which some of the saddle connections are loops bounding a basin of periodic orbits. By using methods from \cite{marmi_cohomological_2005} and \cite{conze_cocycles_2011}, the analysis of ergodic properties on minimal components of such systems boils down to the analysis of skew products studied in this article. 

\subsection{Statement of the main results}

An IET can be fully described through two parameters: a \emph{permutation}, describing the way the intervals are exchanged, and a \emph{length vector}, encoding the lengths of the exchanged intervals. We parametrize the space of (irreducible) IETs of $d \ge 2$ intervals on $[0, 1)$ by indexing the exchanged intervals using an alphabet $\mathcal A$ (with $d$ elements) and considering the set $S_0\times \Lambda^{\mathcal A}$, where $S_0$ denotes the set of \emph{irreducible} permutations (i.e., permutations that do not divide an interval into two or more subintervals which are being exchanged independently) while $\Lambda^{\mathcal A}:= \big\{ \lambda \in \R^\A_+ \,\big|\, |\lambda|_1 = 1\big\}$ stands for the unit simplex of real positive vectors indexed by $\mathcal A$. 

We endow the space $S_0\times \Lambda^{\mathcal A}$ with the product of the counting and the Lebesgue measures. Throughout this work, when using the terms `typical' or `almost every' IET, we mean a.e. IET with respect to this product measure. We refer to Section \ref{sc:notations} for additional background and details on IETs. 


The following is the main result of our article.
\begin{theorem}
\label{thm:main_thm}
Let $m \geq 2$ and $M > 0$. For a.e. irreducible IET $T$ of $d \geq 2$ intervals on $[0, 1)$ the associated cocycle $T_f$ is ergodic, with respect to $\textup{Leb}_{[0, 1)\times \R}$, for a.e. $f \in C_{m, M}$.
\end{theorem}

\begin{remark}
The set of IETs considered in Theorem \ref{thm:main_thm} are the ones satisfying the conclusions of Theorem \ref{thm:good_returns}. The set of cocycles for which the result holds depends on $T$.
\end{remark}

{
Before Chaika and Robertson's work \cite{chaika_ergodicity_2019} concerning the ergodicity of skew-products of the form \eqref{eq:skew_product} for typical cocycles over an exceptional class of IETs known as \emph{linearly recurrent} (we refer to \cite{chaika_ergodicity_2019} for a precise definition), most of the existing results concerned rotations as the base in the definition of the skew product (notice that rotations can be naturally seen as IETs of two intervals). The ergodicity of skew products with rotations in their base was shown for many classes of cocycles, see, for example, the works of Oren \cite{oren_ergodicity_1983}, Hellekalek and Larcher \cite{hellekalek_ergodicity_1986}, Pask \cite{pask_skew_1990}, as well as Conze and Pi\k{e}kniewska \cite{conze_multiple_2014}. However, the case where the base transformation is an IET with more than two intervals has been much less explored. Conze and Fr\k{a}czek \cite{conze_cocycles_2011} proved the ergodicity of skew products with self-similar IETs in the base (which form a countable set of IETs) and piecewise constant functions that are constant over exchanged intervals. Surprisingly, Fr\k{a}czek and Ulcigrai \cite{fraczek_non-ergodic_2014}, during their study of the Ehrenfest wind-tree model, found examples of skew products over IETs with piecewise constant cocycles, again, constant over exchanged intervals, which are not ergodic.

In this work, we extend Chaika and Robertson's result in \cite{chaika_ergodicity_2019} to a full-measure class of IETs (which contains the class of linearly recurrent IETs considered in \cite{chaika_ergodicity_2019}, see Remark \ref{rmk:linearly_recurrent}).

{ Although our arguments resemble those in \cite{chaika_ergodicity_2019}, the different approach that we exhibit allows us to obtain a much more general result. As in \cite{chaika_ergodicity_2019}, we proceed with a proof by contradiction and use crucially the fact that, generically, the discontinuities of the cocycle do not coincide with the discontinuities of the IET and the associated jumps generate a dense subgroup in $\R$. However, we derive a contradiction in a different manner that, in contrast to \cite{chaika_ergodicity_2019}, does not require the direct use of \emph{essential value criteria} (see, e.g., \cite{schmidt_cocycles_1977}) but instead makes ingenious use of some of the properties usually needed for the application of essential value criteria, namely, of the boundedness of Birkhoff sums of the cocycle (see \eqref{eq:Birkhoff_sums}) along return times, to a given set, of the base transformation. 

This approach allows us to provide sufficient conditions on the pairs $(T, f)$ in terms of recurrence properties of $T$ and the Birkhoff sums $S_n f$ for Theorem \ref{thm:main_thm} to hold (see Properties \eqref{cond:return}-\eqref{cond:towers} in Theorem \ref{thm:good_returns}). We then check, in Theorem \ref{thm:good_returns}, that for a.e. irreducible IET there exists a full-measure set of cocycles for which these conditions are satisfied.

A second major difference with respect to \cite{chaika_ergodicity_2019} concerns the methods to control the Birkhoff sums of the cocycle. As mentioned above, bounds on the Birkhoff sums of the cocycle will be essential in the proof and, similarly to the argument in \cite{chaika_ergodicity_2019}, we want these bounds to hold along an \emph{appropriate} sequence of return times (see Properties \eqref{cond:dense} and \eqref{cond:towers} in Theorem \ref{thm:good_returns}) as this will allow us to have similar estimates for a sufficiently large set of perturbations of the initial cocycle (this last fact plays an important role in the proof of Theorem \ref{thm:main_thm}). This kind of control is obtained in \cite{chaika_ergodicity_2019} by combining a quantitative version (proved by the authors in \cite{chaika_ergodicity_2019}) of Atkinson's recurrence theorem \cite{atkinson_recurrence_1976} with properties of the exceptional class of IETs (linearly recurrent) considered by them. We, on the other hand, use the properties of the \emph{Rauzy-Veech renormalization} and the \textit{Kontsevich-Zorich cocycle} (see Section \ref{subs: induction} for precise definitions). More precisely, we use the ergodicity of an accelerated version of the Rauzy-Veech renormalization, known as \emph{Zorich's acceleration} and the fact that the first Lyapunov exponent of the Kontsevich-Zorich cocycle over this acceleration is strictly bigger than the second one, that is, it exhibits a spectral gap (see Section \ref{sc:KZ}). This will imply that certain Rokhlin towers obtained using this algorithm (see Section \ref{sc:RV}) can be balanced in height and width while their heights grow much faster than the Birkhoff sums of the cocycle along the full tower (see Proposition \ref{prop:balanced_times}). As we shall see, this forces many of the Birkhoff sums along shorter subtowers to be bounded (see Proposition \ref{prop:recurrence}). Using this approach, we are able to show the following.}
}

\begin{theorem}
\label{thm:good_returns}
For a.e. irreducible IET $T$ on $d \geq 2$ intervals {there exist $C, \sigma > 0$} such that 
the following holds.

Given $m \geq 2$ and $M > 0$ there exists a full-measure set $\mathcal{F} \subseteq C_{m, M}$ such that, for any $f \in C_{m, M}$, for any positive measure set $E \subseteq [0, 1)$, any $D > mM$ and any $N > 0$, there exist $x \in E$ and $n > N$ satisfying:
\begin{enumerate}
\item \label{cond:return} $x,$ $T^{n}(x) \in E$. 
\item \label{cond:bounded} $\big| S_{n}f(x)\big| < D.$
\item \label{cond:dense} The set $\{T^i(x)\}_{i = 0}^{n- 1}$ is $\tfrac{C}{n}$-dense in $[0, 1)$.
\item \label{cond:towers} {Either $\big[x - \tfrac{\sigma}{n}, x\big]$ or $\big[x, x + \tfrac{\sigma}{n}\big]$ is a continuity interval for $T^{n}$.}
\end{enumerate}
\end{theorem}

\begin{remark}
{It follows from the proof of Theorem \ref{thm:good_returns} that the constants $C, \sigma > 0$ in the statement can be taken uniform for a full-measure set of irreducible IETs on $d$ intervals.} 
\end{remark}

Let us point out that, for any IET $T$ ergodic w.r.t. to $\textup{Leb}_{[0, 1)}$ and any $f \in C_{m, M}$, as a consequence of the recurrence properties of $T_f$, it is not difficult to find times $n$ where Conditions \eqref{cond:return} and \eqref{cond:bounded} are verified. Also, by exploiting the ergodicity of the (accelerated) Rauzy-Veech renormalization, it is not difficult to find, for a.e. $T$ and for some constants $C, \sigma > 0$ not depending on $T$, times $n$ where Conditions \eqref{cond:dense} and \eqref{cond:towers} hold. However, these times do not necessarily coincide. The importance of Theorem \ref{thm:good_returns} is that it guarantees that these times can be made to coincide infinitely many times under generic conditions on $T$ and $f$.

\begin{remark}
\label{rmk:linearly_recurrent}
The conclusions of Theorem \ref{thm:good_returns} are satisfied by the set of \emph{linearly recurrent} IETs introduced in \cite{chaika_ergodicity_2019} (as the last two conditions are trivially verified, for any $n \geq 1$, for some constants $C$, $\sigma$ depending only on $T$). In particular, Theorem \ref{thm:main_thm} also applies to this class of IETs.
\end{remark}


Theorem \ref{thm:good_returns} is proven in the last section of this work and relies heavily on recurrence properties of the accelerated Rauzy-Veech renormalization as well as in well-known results due to Zorich \cite{zorich_deviation_1997} concerning the deviation of certain ergodic averages over IETs (see Theorems \ref{thm:deviation_classical}, \ref{thm:deviations}).

Before describing more in detail the proof strategy of Theorem \ref{thm:main_thm}, let us point out that, as a direct corollary of this theorem, we obtain the following.

\begin{theorem}\label{thm:corollary}
Let $m \geq 2$, $M > 0$ and $d \geq 2$. For a.e. $(T,f)\in S_0\times \Lambda^{\mathcal A}\times C_{m,M}$ the associated skew product $T_f$ is ergodic with respect to $\textup{Leb}_{[0, 1)\times \R}$.

In particular, for a.e. $f\in C_{m,M}$ there exists a full measure set of IETs for which the associated skew product $T_f$ is ergodic with respect to $\textup{Leb}_{[0, 1)\times \R}$. 
\end{theorem}

\begin{proof}
In view of Theorem \ref{thm:main_thm}, it is enough to show that the set 
\[
Z:=\{(T,f)\in S_0\times\Lambda^{\mathcal A}\times C_{m,M} \mid T_f\text{ is ergodic}\}
\]
is measurable.

For any $N\in \N$, let $\tilde T_{f, N}:[0, 1)\times[-N,N] \rightarrow [0, 1)\times[-N,N]$ denote the first return map of $T_f$ to $[0, 1)\times [-N,N]$. Recall that, due to a classical result by Atkinson \cite{atkinson_recurrence_1976}, if $T$ is ergodic, then the associated skew product $T_f$ is \emph{recurrent} (see Appendix \ref{appendix} for a precise definition). In particular the first return transformation $\tilde T_{f, N}$ is well-defined {on a full-measure subset of $[0, 1)\times[-N,N]$.}

Note that $T_f$ is ergodic if and only if $\tilde T_{f, N}$ is ergodic for every $N\in\N$. Indeed, an induced map of an ergodic transformation is also ergodic. {On the other hand, if there was a $T_f$-invariant subset of positive measure whose complement has positive measure, then, for $N \in \N$ sufficiently large, the set $[0, 1) \times [-N,N]$ would contain a $\tilde T_{f, N}$-invariant subset of positive measure whose complement also has positive measure.} 

Thus,
\[
Z=\bigcap_{N=1}^{\infty} Z_N, \qquad \text{ where }Z_N:=\{(T,f)\in S_0\times\Lambda^{\mathcal A}\times C_{m,M} \mid \tilde T_{f, N}\text{ is ergodic}\}.
\]
It is not difficult to show that the assignment $(T, f) \mapsto \widetilde{T}^N_f$ is measurable (see Lemma \ref{lem: contemb} in Appendix \ref{appendix}). Due to the fact that the set of ergodic transformations of any standard probability space is measurable {(see \cite{halmos_lectures_1960})}, it follows that the set $Z_N$ is measurable for every $N\in\N$. Therefore, $Z$ is measurable, which concludes the proof of the theorem.
\end{proof}

\subsection{Scheme of the proof of Theorem \ref{thm:main_thm}}
\label{sc:scheme}
We suppose $T$ is an IET ergodic with respect to the Lebesgue measure and satisfying the conclusions of Theorem \ref{thm:good_returns}.


Suppose, for the sake of contradiction, that for a positive measure set of functions $f$ in $C_{m, M}$, the associated skew products $T_f$ are not ergodic. Call this set $W \subseteq C_{m, M}$. Since a typical function in $C_{m, M}$ has at least two rationally independent jumps (see \eqref{eq:jumps}), we can assume WLOG that for any $f \in W$ the group generated by $\sigma(f)$ is dense in $\R.$

For any $f $ in an appropriate subset $V \subseteq W$, we will \emph{define a $T_f$-invariant measure $\mu_f$}, depending measurably on $f$, that is not a multiple of $\LebIR$, nor singular to it. 

We then \emph{decompose the measure $\mu_f$ along the fibers $\{x\} \times \R$} as a family of measures on $\R$ which we denote $\{\mu_{f, x}\}_{x \in \R}$. Since $\mu_f$ is not a multiple of $\LebIR$, it is possible to show that for a.e. $x \in \R$, the measure $\mu_{f, x}$ is not a multiple of $\Leb$. So we may assume that the set $U$ of $(f, x) \in V \times [0, 1)$ for which $\mu_{f, x}$ is not a multiple of $\Leb$ has full measure.

We will derive a contradiction by showing the existence of $(f, x) \in U$ such that $\mu_{f, x}$ is a multiple of the Lebesgue measure. 

To show this, we first consider cutoffs $\mu_{f, x, D}$ of the measures $\mu_{f, x}$ to compact intervals of the form $[-D, D]$, and \emph{define an appropriate continuity set $B$ for the family of maps $(f, x) \mapsto \mu_{f, x, D}$}, for any $D \in \Q_{> 0}$. We then fix $(f, x) \in B$ and \emph{construct sequences $(f_{i,k}, x^\pm_k)$ in $B$}, converging to $(f, x)$, such that for any open interval $J$ such that $J, J + \sigma_i(f) \subseteq [-D, D]$, we have $$\mu_{f_{i,k}, x^+_k, D}(J) - \mu_{f_{i,k}, x^-_k, D}(J + \sigma_i(f)) \to 0.$$
Using the continuity of these measures on the set $B$, we conclude that $\mu_{f, x}$ is invariant by the translation by $\sigma_i(f)$. Since the set of jumps of $f$ generates a dense subgroup in $\R$, this implies that $\mu_{f, x}$ is a multiple of the Lebesgue measure, which contradicts our assumptions on the set $U$. 

\begin{remark}
The proof of the existence of the measures $\mu_{f}$ described above follows many well-known techniques but is rather technical. For this reason, we provide a detailed proof of this in Appendix \ref{appendix}. 
\end{remark}

\subsection{Outline of the article} First, in Section \ref{sc:notations}, we introduce the notations used throughout the article as well as the main tool in the proof of Theorem \ref{thm:good_returns}, namely, Rauzy-Veech renormalization. 

In Section \ref{sc:main_proof}, assuming Theorem \ref{thm:good_returns}, which describes the set of pairs of IETs and cocycles to be considered, we follow the scheme described in Section \ref{sc:scheme} to prove our main result, namely, Theorem \ref{thm:main_thm}.

In Section \ref{sc:diophantine_condition}, we prove Theorem \ref{thm:good_returns}. Finally, in Appendix \ref{appendix}, we provide some of the technical details required in the proof of Theorem \ref{thm:main_thm}.

\subsection*{Acknowledgments:} We are very grateful to Corinna Ulcigrai and Jon Chaika for fruitful discussions. The first author was supported by the SNSF Grant 188535. The second and third authors acknowledge the support of the Swiss National Science Foundation through Grant 200021\_188617/1. The second author was supported by the NCN Grant 2022/45/B/ST1/00179. The third author was partially supported by the UZH Postdoc Grant, grant no. FK-23-133, the Spanish State Research Agency, through the Severo Ochoa and María de Maeztu Program for Centers and Units of Excellence in R\&D (CEX2020-001084-M) and by the {\it European Research Council} (ERC) under the European Union's Horizon 2020 research and innovation programme Grant 101154283.

\section{Notations}
\label{sc:notations}

\subsection{Interval exchange transformations}
\label{sc:IETs}
We now present the main notions and facts used in this article. Let $I=[0,|I|)$ be an interval of length $|I|$, where $|\cdot|$ denotes the Lebesgue measure of a set. Let also $\mathcal A$ be an alphabet of $d\ge 2$ elements. We say that $T$ is an \emph{interval exchange transformation (IET)} of $d$ elements if there exists a finite partition $\{I_{\alpha}\}_{\alpha\in\A}$ of $I$ into $d$ subintervals such that $T:I\to I$ is a bijection and $T|_{I_\alpha}$ is a translation for every $\alpha\in\mathcal A$. In particular, $T$ preserves the Lebesgue measure. Roughly speaking, $T$ is a permutation of intervals. Since we will be considering $T$ as a measure-preserving system, we will sometimes consider $T$ as a Lebesgue measure-preserving automorphism of $[0, 1)$. 
Then, $T$ is described by two parameters. First is a permutation $\pi=(\pi_0,\pi_1)$, where $\pi_0,\pi_1:\A\to \{1,\ldots,d\}$ are bijections and $\pi_0(\alpha)$ indicates the position of an interval $I_\alpha$ before the exchange, while $\pi_1(\alpha)$ the position of $I_\alpha$ after the exchange. The second parameter is the length vector $\lambda\in \R^{\mathcal A}_+$, where $\lambda_\alpha=|I_{\alpha}|$. We denote $T=(\pi,\lambda)$.

We will often assume some restrictions on $\pi$ and $\lambda$. First, we will only consider $\pi$ to be irreducible, that is there is no $k\in\{1,\ldots,d-1\}$, such that $\pi_1\circ\pi_0^{-1}\{1,\ldots,k\}=\{1,\ldots,k\}$. We denote the set of such permutations as $S_0^{\mathcal A}$. Moreover, we frequently assume that $|I|=1$ and then $\lambda\in\Lambda^{\mathcal A}$, where $\Lambda\subseteq \R^{\mathcal A}_+$ is a positive unit simplex. 

\subsection{Rauzy-Veech induction.}
\label{sc:RV}
We consider an operator $\tilde{R}: S_0^{\mathcal A}\times \R_{>0}^{\mathcal A} \to S_0^{\mathcal A}\times \R_{>0}^{\mathcal A}$, called \emph{Rauzy-Veech induction}. This operator associates to \emph{almost every} IET $(\pi, \lambda)$ defined on an interval $I = [0, |\lambda|)$ (more precisely, to any IET verifying $\lambda_{\pi_0^{-1}(d)}\neq \lambda_{\pi_1^{-1}(d)}$) another IET with the same number of intervals, given by
$$\tilde{R}(\pi,\lambda) := (\pi^{1},\lambda^{1}),$$ where $(\pi^{1},\lambda^{1})$ represents the IET obtained as the first return map of $(\pi,\lambda)$ to the interval 
$$I^{1}:=\left[0,|\lambda|-\min\big\{\lambda_{\pi_0^{-1}(d)},\lambda_{\pi_1^{-1}(d)}\big\}\right).$$ 
We say that $(\pi, \lambda)$ is of \emph{top} (resp. \emph{bottom}) \emph{type} if $\lambda_{\pi_0^{-1}(d)}>\lambda_{\pi_1^{-1}(d)}$ (resp. $\lambda_{\pi_0^{-1}(d)}<\lambda_{\pi_1^{-1}(d)}$).

Let us recall that this first return, which can be defined for \emph{any} IET $(\pi, \lambda)$, determines an interval exchange transformation with the same number of intervals as the initial one if and only if $\lambda_{\pi_0^{-1}(d)}\neq\lambda_{\pi_1^{-1}(d)}$. Moreover, Keane \cite{keane_interval_1975} gave an equivalent condition on $(\pi,\lambda)$ for the iterations of Rauzy-Veech induction to be defined indefinitely. More precisely, we say that an IET $T$ satisfies \emph{Keane's condition} if, for any two discontinuities $a$ and $b$ of $T$, if $T^n(a) = b$ for some $n \in \mathbb{N}$, then $n = 1$, $a = T^{-1}(0)$, and $b = 0$. In particular, if the vector $\lambda$ is \emph{rationally independent}, that is, for any $(c_\alpha)_{\alpha \in \A} \in \Z^\A$,
\[
\sum_{\alpha\in\mathcal A}c_\alpha\lambda_\alpha=0\ \Rightarrow\ c_\alpha=0\text{ for every }\alpha\in\mathcal A,
\]
then $(\pi,\lambda)$ satisfies Keane's condition. 

When the $n$-th iteration of the Rauzy-Veech induction of $(\pi, \lambda)$ is well-defined, we denote
\[(\pi^{n},\lambda^{n}):= \tilde{R}^n(\pi,\lambda),\] 
its domain of definition $[0, |\lambda^n|)$ by $I^{n}$, and its exchanged intervals by $\{I^n_\alpha\}_{\alpha \in \A}$. 

Note that $\lambda^{1}=A (\pi,\lambda)\lambda$, where a matrix $A(\pi,\lambda)$ is defined in the following way
\[
A_{\alpha\beta}(\pi,\lambda)=\begin{cases}
1 & \text{ if }\alpha=\beta;\\
-1 & \text{ if }\alpha=w \text{ and } \beta=l;\\
0 & \text{ otherwise.}
\end{cases}
\]
Denoting $A^{(0)}(\pi, \lambda) = \textup{Id}$, we define, inductively, 
\[
A^{(n)}(\pi,\lambda)=A(\pi^{n-1},\lambda^{n-1})A^{(n-1)}(\pi,\lambda),
\]
for any $n \geq 1$. Then 
\begin{equation}\label{eq: rellambdamatrix}
\lambda^{n}=A^{(n)}(\pi,\lambda)\lambda.
\end{equation}
We will refer to $A^{(n)}(\pi,\lambda)$ as \emph{Rauzy-Veech matrices}. Note that for every $n\in\N$, the matrix ${A^{(n)}(\pi,\lambda)}^{-1}$ is non-negative. One of the many reasons why Rauzy-Veech matrices are important is because the entry $({A^{(n)}(\pi, \lambda)}^{-1})_{\alpha\beta}$ encodes the number of iterates of $I^{n}_\alpha$ contained in $I_\beta$ (sometimes referred to as the \emph{number of visits}) before its first return to $I^{n}$ via $T$. The return time to $I^n$ is the same for all points in $I^n_\alpha$, and we denote it by $q^n_\alpha$. 

Moreover, the vector $(q_{\alpha}^{n})_{\alpha\in\A}$ satisfies 
\begin{equation}
\label{eq:RVheights}
(q_{\alpha}^{n})_{\alpha\in\A} := (A^{(n)})^*(1,\ldots,1),
\end{equation}
where $(\cdot)^*$ denotes inverse transpose matrix. Using this, we can represent $I$ as a disjoint union of $d=\#\A$ Rokhlin towers,
\[I=\bigsqcup_{\alpha\in\A}\bigsqcup_{i=0}^{q_{\alpha}^{n}-1}T^i (I_{\alpha}^{n}),\]
where each `floor' $T^i (I_{\alpha}^{n})$ is an interval. We refer to the numbers $q_{\alpha}^{n}$ as \emph{heigths}, to the vector $(q_{\alpha}^{n})_{\alpha\in\A}$ as the \emph{height} vector, and to the intervals $I_{\alpha}^{n}$ as \emph{bases} of the towers. 

\subsection{Rauzy graphs} The projected action of the Rauzy-Veech induction defines a directed graph structure on $S_0^{\mathcal A}$, where each vertex has exactly two incoming and two outgoing edges. Any invariant subset $\mathfrak{R}\subseteq S_0^{\mathcal A}$ under this projected action is called a \emph{Rauzy graph}.

Note that the value of $A(\pi, \lambda)$ can be derived only from the permutations $\pi$ and $\pi^{1}$. More precisely, $A(\pi, \lambda)$ does not depend on the exact value of $\lambda$ but only on the type of $(\pi, \lambda)$, which in turn determines $\pi^1$. Thus, the matrices Rauzy-Veech matrices of an IET are completely determined by the sequence of permutations corresponding to subsequent iterations of $\widetilde{\mathcal{R}}$. 

Any finite admissible sequence of permutations in the directed graph structure of $\mathfrak{R}$ is called a \emph{Rauzy path}. 
We denote the length of a Rauzy path $\gamma$ by $|\gamma|$. To any Rauzy path $\gamma$, we associate a non-negative matrix $A_\gamma$ given by the product of the inverse Rauzy-Veech matrices along the path $\gamma$. One can show, by considering the set $A_\gamma(\R^{\A}_{+})$, that any Rauzy path $\gamma$ in a Rauzy class $\mathfrak{R}$ is \emph{realizable} as a sequence of permutations obtained by the successive action of $\tilde{R}$, that is, can be expressed as $[\pi^0 \xrightarrow[]{\epsilon(0)} \pi^1 \xrightarrow[]{\epsilon(1)}, \dots, \xrightarrow[]{\epsilon(L - 1)} \pi^L]$, where $\epsilon(i)$ is the type of $(\pi^i, \lambda^i)$, for some $(\pi,\lambda)\in \mathfrak{R}\times \R_+^{\A}$. Hence
\[A_\gamma := {A^{(|\gamma|)}(\pi, \lambda)}^{-1},\]
for any $(\pi, \lambda)$ following $\gamma$ under successive iterations of $\widetilde{\mathcal{R}}$. Whenever the matrix $A_\gamma$ is positive (i.e., all its entries are positive), we say that the path $\gamma$ is \emph{positive}.

Given two Rauzy paths $\gamma$ and $\gamma'$ such that the last permutation in $\gamma$ is the same as the first permutation in $\gamma'$, we denote its \emph{concatenation} by $\gamma \ast \gamma'$. 
 
\subsection{Normalized induction/renormalization}
\label{subs: induction}
The Rauzy-Veech induction facilitates the construction of Rokhlin towers whose properties can be described and, to some extent, controlled through related normalized operators. Given a Rauzy graph $\mathfrak{R}$, we consider the \emph{normalized Rauzy-Veech induction} $\mathcal{R}:\mathfrak{R}\times\Lambda^{\mathcal A} \to \mathfrak{R} \times\Lambda^{\mathcal A}$ given by
\[
\mathcal{R}(\pi,\lambda):= \left(\pi^{1},\frac{\lambda^{1}}{|\lambda^{1}|}\right).
\]
In \cite{masur_interval_1982} and \cite{veech_gauss_1982}, Masur and Veech independently proved that there exists $\mathcal{R}$-invariant measure $\mu_{\mathfrak{R}}$ on $\mathfrak{R}\times\Lambda^{\mathcal A}$, which is equivalent to the product of the counting measure on $\mathfrak{R}$ and the Lebesgue measure on $\Lambda^{\mathcal A}$, such that $(\mathcal{R},\mu_{\mathfrak{R}})$ is an ergodic measure-preserving system. However, the measure $\mu_{\mathfrak{R}}$ is known to be infinite. For this reason, Zorich \cite{zorich_finite_1996} introduced an `acceleration' of $\mathcal{R}$, which we denote by $\mathcal{Z}$, defined on the (full-measure) set of infinitely renormalizable IETs as
\[
\mathcal{Z}(\pi,\lambda):=\mathcal{R}^{\kappa(\pi,\lambda)}(\pi,\lambda),\]
where $$\kappa(\pi,\lambda):=\max\left\{n \geq 1 \mid (\pi,\lambda),\mathcal{R}(\pi,\lambda), \ldots,\mathcal{R}^{n - 1}(\pi,\lambda)\text{ are of the same type}\right\}.$$
 We sometimes refer to the operator $\mathcal{Z}$ as the \emph{Zorich acceleration}. Similarly, we can define a not normalized version of this acceleration, namely 
\[
\widetilde{\mathcal{Z}}(\pi,\lambda):=\widetilde{\mathcal{R}}^{\kappa(\pi,\lambda)}(\pi,\lambda).\]
For any $n\in\N$, we denote
\[(\pi^{(n)},\lambda^{(n)}):= \widetilde{\mathcal{Z}}^n(\pi,\lambda),\]
its domain of definition $[0, |\lambda^{(n)}|)$ by $I^{(n)}$, and its exchanged intervals by $\{I_\alpha^{(n)}\}_{\alpha \in \A}$.

As before, we can express the interval $I = [0, 1)$ as a union of disjoint Rokhlin towers with heights $q^{(n)} = (q_{\alpha}^{(n)})_{\alpha \in \A}$ and bases $\{ I_{\alpha}^{(n)}\}_{\alpha \in \A}$
\[I=\bigsqcup_{\alpha\in\A}\bigsqcup_{i=0}^{q_{\alpha}^{(n)}-1}T^i \big(I_{\alpha}^{(n)}\big),\]
where $q_{\alpha}^{(n)}$ is the return time of $I^{(n)}_\alpha$ to $I^{(n)}$ via $T$. 
 
We will also consider an associated cocycle of matrices, called the \emph{Kontsevich-Zorich cocycle}, defined as $B(\pi,\lambda):=A^{\left(\kappa(\pi,\lambda)\right)}$. We will also use the notation for intermediate matrices 
\begin{equation}
\label{eq:cocycle_intermediate}
B(n,n+1)(\pi,\lambda):= B(\pi^{(n)}, \lambda^{(n)}),
\end{equation}
and the cocycle notation
\begin{equation}
\label{eq:cocycle_notation}
B^{(0)}(\pi, \lambda) = \textup{Id}, \qquad B^{(n)}(\pi,\lambda)=B(\pi^{(n - 1)},\lambda^{(n - 1)})B^{(n-1)}(\pi,\lambda), \quad n \geq 1.
\end{equation}

Denoting
\[Q(\pi, \lambda):= B^*(\pi, \lambda),\]
we can express the lengths and heights of these iterates using cocycle notation
\begin{equation}
\label{eq: rellambdamatrixacceleration}
\lambda^{(n)}=B^{(n)}(\pi,\lambda)\lambda, \qquad q^{(n)} = Q^{(n)}(\pi,\lambda)(1, \dots, 1),
\end{equation}
{
where $Q^{(n)}(\pi,\lambda)$ is defined in a manner analogous to \eqref{eq:cocycle_notation}.

We define the intermediate steps $Q(n, n + 1)(\pi, \lambda)$ in a manner similar to \eqref{eq:cocycle_intermediate}. Moreover, if there is no risk of confusion, we will sometimes omit the IET being considered in the notations for intermediate steps of the cocycles $B$ and $Q$ and denote, for example, $Q(n, n + 1)(\pi,\lambda)$ simply by $Q(n, n + 1)$.}

The crucial difference between $\mathcal{Z}$ and $\mathcal{R}$ is that the former preserves a \emph{finite} measure $\rho_{\mathfrak{R}}$, as proven by Zorich in \cite{zorich_finite_1996}. 

\subsection{Lyapunov exponents of the Kontsevich-Zorich cocycle}
\label{sc:KZ}
In \cite{zorich_finite_1996}, Zorich proved that the Kontsevich-Zorich cocycle $B$ (and, therefore, also $Q$) is \emph{integrable}. More precisely, he proved that
\begin{equation}\label{eq: integrability}
\int_{\mathfrak{R}\times\Lambda^{\mathcal A}} \log \|B(\pi,\lambda)\|\,d\rho_{\mathfrak{R}}, \int_{\mathfrak{R}\times\Lambda^{\mathcal A}} \log \|Q(\pi,\lambda)\|\,d\rho_{\mathfrak{R}} <\infty.
\end{equation}
The above fact is one of the crucial tools used in this article since it allows us to consider the Lyapunov exponents associated with these cocycles. Notice that since $B$ and $Q$ are \emph{dual} to each other (we refer the interested reader to \cite{zorich_deviation_1997} for a precise definition of dual cocycle), they have the same Lyapunov spectrum when restricted to any ergodic component of $\mathcal{Z}$.

Let us recall that, as a consequence of several classical works \cite{veech_gauss_1982, zorich_finite_1996, forni_deviation_2002, avila_simplicity_2007}, for any Rauzy class $\mathfrak{R}$, the Zorich acceleration is ergodic when restricted to $\mathfrak{R} \times \Lambda^\A$ and that, for some $g \geq 1$ depending only on $\mathfrak{R}$, the cocycles $B$ and $Q$ restricted to $\mathfrak{R} \times \Lambda^\A$ possess $2g$ non-zero Lyapunov exponents of the form
\[ -\theta_{1} < \dots < -\theta_{g} < 0 < \theta_g < \dots < \theta_1,\]
while the remaining $d - 2g$ Lyapunov exponents are equal to $0$.

{For more information on interval exchange transformations and associated renormalizations, we refer the reader to \cite{viana_ergodic_2006} and \cite{yoccoz_interval_2010}.}
{
\subsection{Path-induced domains} We prove now a series of simple results concerning $\mathcal{R}$ (resp. $\mathcal{Z}$) and the associated Rauzy-Veech (resp. Kontsevich-Zorich) cocycle. The main purpose of these results is to define appropriate domains on which the first return map of $\mathcal{R}$ (resp. $\mathcal{Z}$) and the associated cocycle display certain properties that we describe below. This will be particularly useful in the proof of Proposition \ref{prop:balanced_times}, which is a crucial tool in Theorem \ref{thm:good_returns}.

Let $\mathfrak{R}$ be a Rauzy graph and let $\gamma$ be a Rauzy path in $\mathfrak{R}$ of length $L:=|\gamma|$. Define
\[
\Delta_{\gamma}:=\pi_{\Lambda^{\mathcal A}} \big( A_{\gamma}(\R_+^{\A})\big) =\left\{(\pi,\lambda)\in\mathfrak{R}\times \Lambda^{\A}\mid (\pi,\lambda)\text{ follows $\gamma$ via }\mathcal{R}\right\},
\]
where $\pi_{\Lambda^{\mathcal A}}:\R_+^{\A}\to \Lambda^{\mathcal A}$ is {the projection $v \mapsto \tfrac{v}{|v|_1}$ onto} the simplex $\Lambda^{\mathcal A}$. 

By extending the path $\gamma$, we may guarantee that we do not see path $\gamma$ as an initial Rauzy path for the first $L$ iterations of $\mathcal{R}$.
\begin{lemma}\label{lem: noreturn}
Let $\mathfrak{R}$ be a Rauzy class and let $\gamma$ be a Rauzy path in $\mathfrak{R}$ of length $L \geq 1$. There exists a path $\tilde\gamma$ of length $2L$ such that $\Delta_{\tilde\gamma}\subseteq \Delta_{\gamma}$ and 
 \[
\mathcal{R}^i(\Delta_{\tilde\gamma})\cap \Delta_{\gamma}=\emptyset\quad\text{for}\quad i=1,\ldots,L-1.
 \]
\end{lemma}
\begin{proof}
WLOG suppose that the last arrow in the path $\gamma$ was of bottom type; the proof in the opposite case is analogous. Take the path $\tilde\gamma:=\gamma * \gamma_L^t$, where $\gamma_L^t$ is a path of length $L$, starting at the final vertex of $\gamma$, such that all its arrows are of top type. Since $\tilde\gamma$ starts with $\gamma$, we have $\Delta_{\tilde\gamma}\subseteq \Delta_{\gamma}$. Moreover, for any $(\pi,\lambda)\in\Delta_{\tilde\gamma}$ and any $i=1,\ldots,L-1$, we have that $\mathcal{R}^{i+L}(\pi, \lambda)$ is of top type. Since the last arrow in $\gamma$ is of bottom type, this finishes the proof of the lemma.
\end{proof}

Recall that if $(\pi,\lambda)$ is infinitely renormalizable then there exists $m\in\N$ such that the path $\gamma_m(\pi, \lambda)$ obtained by considering the first $m$ iterations of $(\pi, \lambda)$ by $\mathcal{R}$ is positive. Moreover, for a.e. $(\pi,\lambda)$ we have
\begin{equation}
\label{eq:uniquely_ergodic}
\{\lambda \} = \bigcap_{m \geq 0} \Delta_{\gamma_m(\pi, \lambda)}.
\end{equation}
For proof of these well-known properties, we refer the interested reader to \cite{viana_ergodic_2006} (see Corollary 5.3 and Remark 29.3).

The following well-known fact, which we prove here for completeness, concerns the relation between the positivity of this matrix and the height vectors associated with the Rauzy-Veech induction.

\begin{lemma}\label{lem: heigthsandmatrices}
Let $\mathfrak{R}$ be a Rauzy class and let $\gamma$ be a positive Rauzy path in $\mathfrak{R}$. There exists a constant $C_\gamma>0$ such that for any infinitely renormalizable $(\pi,\lambda)\in\mathfrak{R}\times\Lambda^{\mathcal A}$ and any $K\ge L \geq 0$ satisfying $\mathcal{R}^{K-L}(\pi,\lambda)\in \Delta_{\gamma}$, 
 \[
 \max_{\alpha,\beta\in\mathcal A}\,\frac{q_{\alpha}^{K}}{q_{\beta}^{K}}\le C_{\gamma}.
 \]
\end{lemma}
\begin{proof}
By \eqref{eq:RVheights}, 
 \[
[q^{K}_{\alpha}]_{\alpha\in\A}=[q^{K-L}_{\alpha}]_{\alpha\in\A}\cdot A_{\gamma}.
 \]
Since $\gamma$ is positive, we also have $\min_{\alpha\in\A}q^{K}_{\alpha}>\max_{\alpha\in\A}q^{K-L}_{\alpha}.$
 Hence 
 \[
\max_{\alpha,\beta\in\mathcal A}\,\frac{q_{\alpha}^{K}}{q_{\beta}^{K}}\le 
\frac{\max_{\alpha\in\A}q^{K-L}_{\alpha}\cdot \sum_{\alpha,\beta\in\mathcal A}(A_\gamma)_{\alpha\beta}}{\max_{\alpha\in\A}q^{K-L}_{\alpha}}.
 \]
 Thus, to conclude the proof, it is enough to take $C_{\gamma}:=\sum_{\alpha,\beta\in\mathcal A}(A_\gamma)_{\alpha\beta}$.
\end{proof}
For convenience, in the proof of Proposition \ref{prop:balanced_times}, we consider a path that is the concatenation $\gamma*\gamma*\gamma$ of 3 copies of the path $\gamma$. It is easy to check that $\Delta_{\gamma*\gamma*\gamma}=A_{\gamma}^3(\R^{\mathcal A}_+)$. However, we will require more from our path. In the next lemma, we will construct a path $\gamma$, which depends on the starting point and is coherent with the Kontsevich-Zorich acceleration.
\begin{lemma}\label{lem: constructionofgamma}
Let $\mathfrak{R}$ be a Rauzy class and fix $\ell_0 \in \N$. For a.e. $(\pi,\lambda)\in \mathfrak{R}\times\Lambda^{\mathcal A}$, there exists $\ell > \ell_0$ such that the path $\gamma = [\pi^0 \xrightarrow[]{\epsilon(0)}, \dots, \xrightarrow[]{\epsilon(\ell - 1)} \pi^\ell]$ of length $\ell$, associated with the first $\ell$ iterations of $(\pi, \lambda)$ by $\mathcal{R}$, satisfies:
 \begin{itemize}
 \item $A_{\gamma}$ is a positive matrix,
 \item $\pi^0 = \pi = \pi^\ell$,
 \item the first and last arrow of $\gamma$ are of opposite types.
 \end{itemize}
 In particular, if $(\pi,\lambda')\in \Delta_{\gamma*\gamma*\gamma*\varphi}$, where 
 $\varphi = [\pi^0 \xrightarrow[]{\epsilon(0)} \pi^1]$, then
 \[
\mathcal{R}^{i\ell}(\pi,\lambda')=\mathcal{Z}^{iL}(\pi,\lambda') \text{ for }i=1,2,3,
 \]
 where $L$ is the number of type changes in the path $\gamma$.
\end{lemma}
\begin{proof}
 Since $\mathcal{R}$ is ergodic, we may assume that $(\pi,\lambda) \in \mathfrak{R}\times\Lambda^{\mathcal A}$ is a generic point for $\mathcal{R}$. WLOG, we may also assume that $\mathcal{R}(\pi,\lambda)$ is of top type, as the other case is completely analogous. Since $(\pi, \lambda)$ is infinitely renormalizable, there exists $\tilde \ell > \ell_0$ such that the path $\gamma_{\tilde \ell}(\pi, \lambda)$ is positive. Notice that, for any $\ell \geq \tilde \ell$, the path $\gamma_{\ell}(\pi, \lambda)$ is also positive since the product of a positive matrix with a non-negative matrix having positive diagonal entries yields a positive matrix. 
 
 Consider an arrow of bottom type that ends at $\pi$. Let $\pi'$ be the initial vertex of this arrow. Then, for a positive measure set of length vectors $\hat\lambda$, we have that $\mathcal{R}(\pi',\hat\lambda)$ is of bottom type. Hence, by ergodicity of $\mathcal{R}$, there exists $\ell \ge\tilde \ell$ such that $\mathcal{R}^{\ell - 1}(\pi, \lambda) \in \{\pi'\} \times \Lambda^\A$ is of bottom type, and the path $\gamma$ associated with the first $\ell$ iterates of $(\pi,\lambda)$ by $\mathcal{R}$ satisfies the desired conditions. 
 
Notice that the last assertion of the lemma follows from the definition of the Kontsevich-Zorich acceleration and the fact that the first and the last arrow of $\gamma$ are of opposite types.
\end{proof}

Finally, we state a result that will be helpful in the proof of Proposition \ref{prop:balanced_times}.

\begin{proposition}\label{prop: pathconstruction}
Let $\mathfrak{R}$ be a Rauzy class. For a.e $(\tilde\pi,\tilde\lambda)\in\mathfrak{R}\times \Lambda^{\A}$ and for any open neighbourhood $(\tilde\pi,\tilde\lambda)\in U\subseteq \mathfrak{R}\times \Lambda^{\A}$ there exist paths $\gamma, \tilde \gamma$ in $\mathfrak{R}$ such that $\Delta_{\tilde \gamma} \subseteq \Delta_{\gamma} \subseteq U$ and, denoting $\ell := |\gamma|$, for a.e. $(\pi,\lambda)\in \Delta_{\tilde\gamma}$ the following conditions hold. 
 \begin{itemize}
\item All the entries of $A_\gamma$ are bigger or equal than $2$.
\item $(\pi,\lambda)$ is infinitely renormalizable.
\item The next $3\ell$ Rauzy-Veech renormalizations of $(\pi,\lambda)$ follow the concatenated path $\gamma \ast \gamma \ast \gamma$. 
\item $\mathcal{R}^{i\ell}(\pi, \lambda) = \RVKZ^{i L}(\pi, \lambda)$, for $i \in \{1, 2, 3\}$.
\item $\RVKZ^i(\pi, \lambda) \notin \Delta_{\tilde \gamma}$, for $i \in \{1, \dots, 3L - 1\}$.\item $\Delta_{\tilde \gamma} \cup \RVKZ^{L}(\Delta_{\tilde \gamma}) \cup \RVKZ^{2L}(\Delta_{\tilde \gamma}) \subseteq \mathfrak{R} \times U.$
\item 

\begin{equation}
\label{eq:several_balanced_heights}
\max_{\substack{\alpha, \beta \in \A, \\ i \in \{1, 2, 3\}}} \frac{q^{(iL)}_\alpha(\pi, \lambda)}{q^{(iL)}_\beta(\pi, \lambda)} < C_\gamma.
\end{equation}
\end{itemize}
In particular, if $U = \left\{ \lambda \in \Lambda^\A \,\left|\, |\lambda_\alpha - \lambda_\beta| \leq \tfrac{\nu}{2d}, \quad \forall \alpha, \beta \in \A, \alpha \neq \beta\right\}\right.$ for some $0 < \nu < 1$, then 
\begin{equation}
\label{eq:several_balanced_lengths}
\max_{\substack{\alpha, \beta \in \A, \\ i \in \{0, 1, 2\}}} \frac{\lambda^{(iL)}_\alpha(\pi, \lambda)}{\lambda^{(iL)}_\beta(\pi, \lambda)} < 1 + \nu.
\end{equation}
\end{proposition}
\begin{proof}
Let $\ell$, $\gamma = \gamma_{\ell}(\tilde\pi,\tilde\lambda)$ and $\varphi$ as in Lemma \ref{lem: constructionofgamma} when applied to $(\tilde\pi,\tilde\lambda)$. Assuming WLOG that $(\pi, \lambda)$ satisfies \eqref{eq:uniquely_ergodic} and taking $\ell$ larger if necessary, we may assume that all the entries of $A_\gamma$ are bigger or equal to $2$ (since for $\ell$ sufficiently large we can see $\gamma_{\ell}(\tilde\pi,\tilde\lambda)$ as a concatenation of two positive paths and thus see $A_\gamma$ as the product of two positive matrices) and that $\Delta_\gamma\subseteq U$. 

Applying Lemma \ref{lem: noreturn} to $\gamma*\gamma*\gamma * \varphi$ yields a positive path $\tilde \gamma$ (obtained as a concatenation of $\gamma*\gamma*\gamma * \varphi$ with another path of equal length) which, by Lemmas \ref{lem: noreturn}, \ref{lem: heigthsandmatrices} and \ref{lem: constructionofgamma}, verifies all of the desired conditions.
\end{proof}

}

\subsection{Nudging of discontinuities}

{This notion was first introduced in the work of Chaika and Robertson \cite{chaika_ergodicity_2019}.}
Given $f \in C_{m, M}$ of the form $p_1\chi_{[0, q_1)} + \dots + p_{m + 1}\chi_{[q_1 + \dots + q_m, 1)}$ for some some $(q, p) \in \R^{m + 1}_+ \times \R^{m + 1}$, we can move the location of its $i$-th discontinuity by $\zeta \in \R$ provided that $|\zeta| < \tfrac{\Gamma}{2}$, where $\Gamma := \min \{q_1, \dots, q_{m +1}\}$, by considering the function $\textup{nudge}(f, i, \zeta) \in C_{m, M}$ associated to the vector
\[ \left(q_1, \dots, q_{i - 1}, q_i + \zeta, q_{i + 1} - \zeta, q_{i + 1}, \dots, q_{m + 1}, p_1, \dots, p_i, \frac{p_{i +1}q_{i + 1} - \zeta p_i}{q_{i + 1} - \zeta}, p_{i + 2}, \dots, p_{m + 1} \right).\]
Notice that the location of other discontinuities, as well as $m - 1$ of the possible values taken by $f$, remain unchanged. Moreover, we have
\[ \left| \frac{p_{i +1}q_{i + 1} - \zeta p_i}{q_{i + 1} - \zeta}, - p_{i + 1} \right| = \left| \frac{\zeta p_{i +1} - \zeta p_i}{q_{i + 1} - \zeta}\right| \leq \frac{2M|\zeta|}{|q_{i + 1} - \zeta|} \leq \frac{4M|\zeta|}{q_{i + 1}} \leq \frac{4M|\zeta|}{\Gamma},\]
which implies the following.
\begin{lemma}[Lemma 4.3 in \cite{chaika_ergodicity_2019}]
\label{lemma:nudging}
For any $f=p_1\chi_{[0,q_1)}+...+p_{m + 1}\chi_{[q_1+\dots +q_{m},1)} \in C_{m, M}$ and any $\zeta \in \R$ verifying $|\zeta| < \tfrac{1}{2}\min \{q_1, \dots, q_{m +1}\}$, 
\[ |\zeta| \leq |f - \textup{nudge}(f, i, \zeta)| \leq |\zeta| \max \left\{1, \frac{4M}{\min \{q_1, \dots, q_{m +1}\} } \right\}. \]
\end{lemma}

\section{Proof of Theorem \ref{thm:main_thm}}
\label{sc:main_proof}

We will follow the strategy described in Section \ref{sc:scheme}. Let $T$ be an IET ergodic with respect to the Lebesgue measure, satisfying the conclusions of Theorem \ref{thm:good_returns}, for some $C, \delta > 0$. Notice that by Theorem \ref{thm:good_returns}, this is satisfied by a full-measure set of IETs.

 Suppose, for the sake of contradiction, that for a positive measure set $W$ of functions $f$ in $C_{m, M}$, the associated skew products $T_f$ are not ergodic. We assume WLOG that, for any $f \in W$, the group generated by $\sigma(f)$ is dense in $\R$ and that the conclusions of Theorem \ref{thm:good_returns} hold for $f$.
 
\subsection*{Step 1 - Construction and decomposition of a family of $T_f$-invariant measures} For any $f$ in an appropriate positive measure subset $V \subseteq W$, we will define a measure $\mu_f$ on $[0, 1) \times \R$, depending in a measurable way on $f$, such that
\begin{itemize}
\item $\mu_f$ is invariant by $T_f$;
\item $\mu_f \neq c\LebIR,$ for any $c \in \R \setminus \{0\}$;
\item $\mu_f \not\perp \LebIR$.
\end{itemize}

 Our construction relies on many well-known techniques and schemes, but it is rather technical. For the sake of clarity and completeness, we provide a complete proof of the existence of the measures described above in Proposition \ref{prop:measurable_dependence_measures} of Appendix \ref{appendix}.

\subsection*{Step 2 - Decompose the measures $\mu_f$ on the fibers}
For each $f \in V$, we decompose $\mu_f$ with respect to the partition of $[0, 1) \times \R$ in vertical lines and denote this family of measures by $\{\mu_{f,x}\}_{x\in [0, 1)}$, where $\mu_{f, x}$ is supported on $\{x\} \times \R$. Notice that the map $(f, x) \mapsto \mu_{f,x}$ is measurable with respect to $f$ and $x$. Indeed, it is formed by gluing the pieces obtained from standard disintegration over the first coordinate on the sets of the form $[0, 1)\times [N,N+1]$, $N\in\Z$. Moreover, since $\mu_f$ is $T_f$-invariant and by the uniqueness of this decomposition, 
\begin{equation}
\label{eq:invariance_decomposition}
\mu_{f, T^k(x)} = (T_f^k)_* (\mu_{f,x}),
\end{equation}
for a.e. $x$ and any $k \in \Z$. As an abuse of notation, in the following, we will treat the measures $\{\mu_{f,x}\}_{x\in [0, 1)}$ as measures on $\R$. With this convention, equation \eqref{eq:invariance_decomposition} becomes
\begin{equation}
\label{eq:invariance_decomposition_BS}
\mu_{f, T^k(x)} = (V_{S_kf(x)})_* (\mu_{f,x}),
\end{equation}
where $S_k f(x)$ denotes the $k$-th Birkhoff sum of $f$ with respect to $T$ and 
\begin{equation}
\label{eq:translation}
\Function{V_r}{\R}{\R}{t}{t + r},
\end{equation}
for any $r \in \R$.

We will assume WLOG that, for any $f \in V$,
\begin{itemize}
\item $\mu_{f, x}$ is not a multiple of $\Leb$, for a.e. $x \in [0, 1).$
\end{itemize}
Indeed, if there exists a positive measure set of points in $[0, 1)$ for which $\mu_{f, x} = c_x \Leb,$ for some constants $c_x \in \R$, then, by \eqref{eq:invariance_decomposition} and the ergodicity of $T$ with respect to the Lebesgue measure, it follows that $\mu_{f, x} = c_x \Leb$, for a.e. $x \in [0, 1)$. Since the function $x \mapsto c_x$ is $T$-invariant, by the ergodicity of $T$ with respect to the Lebesgue measure, there exists $c \in \R$ such that $c_x = c$, for a.e. $x \in [0, 1)$. Therefore, $\mu_f = c\LebIR$, which contradicts our initial assumption on $\mu_f$.

 Note that, since the map $(f,x)\mapsto\mu_{f,x}$ is measurable, the set $U = \{(f, x) \in V \times [0, 1) \mid \mu_{f, x} \neq c\Leb,\, \forall c \in \R \}$ is measurable as well.
\subsection*{Step 3 - Consider an appropriate continuity set for the (cutoffs of) the map $(f, x) \mapsto \mu_{f, x}$}
For $L>0$, let $\mu_{f,x,L}$ be the normalized cut-off of the measure of $\mu_{f, x}$ to $[-L, L]$. By Lusin's Theorem, there exists a compact set $K \subseteq U$ of positive measure such that, for any $D \in \Q \cap (0, +\infty)$,
\begin{itemize}
\item $\displaystyle\begin{aligned}[t] \Function{\Phi_D}{K \subseteq U}{\mathcal{M}_L}{(f, x)}{\mu_{f,x,L}} \end{aligned}$ is continuous,
\end{itemize}
where $\mathcal{M}_D$ denotes the space of probability measures on $[-L, L].$ 

{ By the Lebesgue density theorem, we may assume WLOG that every point in $K$ is a density point of $K$. Moreover, we may assume that every point in $K$ is a \emph{fiber-wise density point }, that is, for any $(x, f) \in K$, denoting \[ E_f := \{y \in [0, 1) \mid (f, y) \in K\}, \quad P_x := \{ g \in C_{m,M} \mid (g, x) \in K\}, \]
we may assume that $x$ (resp. $f$) is a density point of $E_f$ (resp. $P_x$). { 
Indeed, we have the following fact. 
\begin{lemma}
    For every measurable set $A\in\R^k\times \R^m$ of positive Lebesgue measure and almost every point $(x,y)\in A$, the points $x\in\R^k$ and $y\in \R^m$ are density points w.r.t. Lebesgue measures on corresponding fibers.
\end{lemma}
\begin{proof}
    We prove the lemma for $n=m=1$; the other cases follow analogously. Moreover, by lower regularity of the Lebesgue measure, we may assume that $A$ is compact, and to simplify the notation, we assume that $A\subset[0,1]^2$. We will show that for a.e. $(x,y)\in A$, the point $x$ is a density point w.r.t. the Lebesgue measure on $[0,1]\times\{y\}$, the other case being treated identically.

    Consider the measure
    \[
\lambda:=Leb|_{A}\quad\text{on}\ [0,1]^2.
    \]
By the Disintegration Theorem, there exists a measurable assignment $y\mapsto\lambda_{y}$ such that 
\[
\lambda=\int_0^1\lambda_y\,d\lambda',
\]
where $\lambda'$ is a projection of $\lambda$ on $y$-coordinate. Moreover, by the definition of $\lambda$, $\lambda_y:=Leb_{[0,1]\times \{y\}}$ for$\lambda'$-a.e. $y$. We know by the Lebesgue density theorem that for $\lambda'$-a.e. $y$,  $\lambda_y$-a.e. $x$ is a Lebesgue density point for the set $A\cap[0,1]\times \{y\}$ (w.r.t. the Lebesgue measure on $[0,1]\times\{y\}$). Hence, by Fubini's Theorem, to conclude the proof it is enough to show that the set of points
\[
\tilde A:=\{(x,y)\in A\mid \text{$x$ is a Lebesgue density point for the set $A\cap[0,1]\times \{y\}$}\}
\]
is measurable.

Define a map $F_\epsilon:A\to [0,1]$ in the following way 
\[
F_{\epsilon}(x,y):=\frac{1}{2\epsilon}\lambda_{y}(x-\epsilon,x+\epsilon).
\]
Note that $F_{\epsilon}$ is measurable since it is a composition of two measurable maps. Then let
\[
F(x,y):=\liminf_{\epsilon\to 0}F_{\epsilon}(x,y).
\]
As a lower limit of measurable functions, $F$ is also measurable. To finish the proof, it remains to notice that $\tilde A=F^{-1}(\{1\})$.
\end{proof}
}

Uniform fiber-wise density estimates (for positive measure subsets of $K$) will prove very helpful later in the proof. For any subset $G \subseteq K$ and any $f \in C_{m, M}$, we denote
\[ E_f^G := \{y \in [0, 1) \mid (f, y) \in G\}.\]

By Egorov's theorem, together with the Lebesgue density theorem, we can find a positive measure subset $B_V \subseteq K$ such that:
\begin{itemize}
\item For any $0 < \xi < 1$, there exists $N_\xi \in \N$ such that,  for any $n>N_\xi$ and any $(f, x) \in B_V$,
\begin{equation}
\label{eq:quantitative_density_1}
 \frac{\textup{Leb}_{C_{m,M}}\{g\in P_x \mid d(f,g)<\su{n}\}}{\textup{Leb}_{C_{m,M}}\{g\in C_{m,M} \mid d(f,g)<\su{n}\}}> \xi. 
\end{equation}
\end{itemize}

In the same way as before, we may assume WLOG that every point in $B_V$ is a fiber-wise density point of $B_V.$ Thus, again by Egorov's theorem together with the Lebesgue density theorem, we can find a positive measure subset $B_H \subseteq B_V$ such that:

\begin{itemize}
\item For any $0 < \xi < 1$, there exists $N_\xi \in \N$ such that,  for any $n>N_\xi$ and any $(f, x) \in B_H$,
\begin{equation}
\label{eq:quantitative_density_2}
\frac{\textup{Leb}_{[0, 1)}\{y\in E_f^{B_V} \mid |x - y| <\su{n}\}}{\tfrac{2}{n}}> \xi.
\end{equation}
\end{itemize}

Similarly, we can find a positive measure subset $B \subseteq B_H$ such that:

\begin{itemize}
\item For any $0 < \xi < 1$, there exists $N_\xi \in \N$ such that,  for any $n>N_\xi$ and any $(f, x) \in B$,
\begin{equation}
\label{eq:quantitative_density_3}
\frac{\textup{Leb}_{[0, 1)}\{y\in E_f^{B_H} \mid |x - y| <\su{n}\}}{\tfrac{2}{n}}> \xi.
\end{equation}
 \end{itemize}
We assume WLOG that every point in $B$ is a fiber-wise density point of $B$.
}

\subsection*{Step 4 - Find good returns to the set $E_f$} Fix $(f, x) \in B$ and $D > mM$. We can define sequences {$(x_k)_{k \geq 1} \subseteq E_f^{B_H}$} and $(r_k)_{k \geq 1} \subseteq \N$, with $r_k \nearrow +\infty$, such that:

\begin{enumerate}
\item $x_k,$ $T^{r_k}(x_k) \in { E_f^{B_H}} \cap \big(x - \tfrac{1}{k}, x + \tfrac{1}{k}\big)$,
\item $\big|S_{r_k}f(x_k)\big| < D,$
{ \item $\{T^i(x_k)\}_{i = 0}^{r_k - 1}$ is $\tfrac{C'}{r_k}$-dense in $[0, 1)$,
\item $\big(x_k- \tfrac{\sigma'}{r_k}, x_k + \tfrac{\sigma'}{r_k}\big)$ is a continuity interval for $T^{r_k}$,}
\end{enumerate}
{where $C' := C + 1,$ $\sigma' := \tfrac{\sigma}{4}$ and $C, \sigma > 0$ are the constants for which the conclusions of Theorem \ref{thm:good_returns} hold. 

Indeed, let $N_0 \geq 1$. For any $k \geq 1$, Theorem \ref{thm:good_returns} applied to {$E_k = E_f^{B} \cap \big(x - \tfrac{1}{2k}, x + \tfrac{1}{2k}\big)$} with $N = 2k + N_0$ yields a point $y_k \in {E_f^{B}}\cap \big(x - \tfrac{1}{2k}, x + \tfrac{1}{2k}\big)$ and a natural number $r_k > N$ such that $y_k, T^{r_k}(y_k) \in E_k$, $\big|S_{r_k}f(y_k)\big| < D$, $\{T^i(y_k)\}_{i = 0}^{r_k - 1}$ is $\tfrac{C}{r_k}$-dense in $[0, 1)$, and either $\big[y_k - \tfrac{\sigma}{r_k}, y_k\big]$ or $\big[y_k, y_k + \tfrac{\sigma}{r_k}\big]$ is a continuity interval of $T^{r_k}$. Let $J_k$ denote this continuity interval. 

By \eqref{eq:quantitative_density_3}, up to taking $N_0$ sufficiently large, we have
\begin{equation}
\label{eq:big_intersection}
|J_k \cap { E_f^{B_H}}|, |T^{r_k}(J_k) \cap { E_f^{B_H}}| > \tfrac{3|J_k|}{4}.
\end{equation}
Notice that, since $J_k, T^{r_k}(J_k) \subseteq \big(x - \tfrac{1}{k}, x + \tfrac{1}{k}\big)$, we have
\[J_k \cap { E_f^{B_H}} = J_k \cap { E_f^{B_H}} \cap \big(x - \tfrac{1}{k}, x + \tfrac{1}{k}\big)\qquad T^{r_k}(J_k) \cap { E_f^{B_H}} = T^{r_k}(J_k) \cap { E_f^{B_H}} \cap \big(x - \tfrac{1}{k}, x + \tfrac{1}{k}\big),\]
and thus equation \eqref{eq:big_intersection} yields
\[ |J_k \cap { E_f^{B_H}} \cap \big(x - \tfrac{1}{k}, x + \tfrac{1}{k}\big) \cap T^{-r_k}({ E_f^{B_H}} \cap \big(x - \tfrac{1}{k}, x + \tfrac{1}{k}\big))| > \tfrac{|J_k|}{2}.\]
Therefore, there exists $x_k \in \tfrac{1}{2}J_k$, where $ \tfrac{1}{2}J_k$ denotes the centered interval inside $J_k$ of length $\tfrac{1}{2}|J_k|$, satisfying the first two assertions above. 

Since $\{T^i(x_k)\}_{i = 0}^{r_k - 1}$ is simply a translation of $\{T^i(y_k)\}_{i = 0}^{r_k - 1}$ by at most $\tfrac{\sigma}{r_k}$, $x_k$ verifies the third assertion. Finally, noticing that $\big(x_k- \tfrac{\sigma}{4r_k}, x_k + \tfrac{\sigma}{4r_k}\big) \subseteq J_k$, it follows that $x_k$ also verifies the fourth assertion.
}

\subsection*{Step 5 - Construct appropriate perturbations of $f$} We will define $\xi_1 > \xi_2 > 0$ and sequences $(g_{i, k})_{k \geq 1} \subseteq C_{m, M}$ of perturbations of $f$, for $i = 1, \dots, m$, such that, for any $k \geq 1$ and any $i = 1, \dots, m,$
\begin{enumerate}
\item $|g_{i, k} - f| < \frac{\xi_1}{r_k}.$ 
\item For any $g \in C_{m, M}$ verifying $|g - g_{i, k}| < \frac{\xi_2}{r_k},$ the $i$-th discontinuity of $g$ belongs to the Rokhlin tower $\bigsqcup_{j = 0}^{r_k - 1}T^j\big(x_k- \tfrac{\sigma'}{2r_k}, x_k + \tfrac{\sigma'}{2r_k}\big)$. Moreover, no other discontinuity of $h$ belongs to $\bigsqcup_{j = 0}^{r_k - 1}T^j\big(x_k- \tfrac{3\sigma'}{4r_k}, x_k + \tfrac{3\sigma'}{4n_k}\big).$
\end{enumerate}
We will do it by nudging, as defined in Lemma \ref{lemma:nudging}. This will allow us to control the relative position of the discontinuities inside the chosen Rokhlin towers.


Fix $1 \leq i \leq m$. By nudging $f$ using Lemma \ref{lemma:nudging}, let us construct a function $g_{i, k} \in C_{m,M}$ such that its $i$-th discontinuity belongs to $\bigsqcup_{j = 0}^{r_k - 1}T^j\big(x_k- \tfrac{\sigma'}{4r_k}, x_k + \tfrac{\sigma'}{4r_k}\big)$ and no other of its discontinuities belong to $\bigsqcup_{j = 0}^{r_k - 1}T^j\big(x_k- \tfrac{4\sigma'}{5r_k}, x_k + \tfrac{4\sigma'}{5r_k}\big).$

Recall that any $f \in C_{m, M}$ is identified with a vector $(q, p) \in \R^{m + 1} \times \R^{m + 1}$. Let us denote $\Gamma:= \min \{q_1, \dots, q_{m +1}\}$. By nudging $f$ at most $m$ times by a distance of at most $\tfrac{4\sigma'}{5r_k}$ (which we assume WLOG to be smaller than $\tfrac{\Gamma}{4}$), we can define a function $\overline{g}_k \in C_{m, M}$ having no discontinuities in $\bigsqcup_{j = 0}^{r_k - 1}T^j\big(x_k- \tfrac{4\sigma'}{5r_k}, x_k + \tfrac{4\sigma'}{5r_k}\big)$ and such that $|f - \overline{g}_k| \leq \tfrac{4m\sigma'}{5r_k}\tfrac{{4M + 1}}{\Gamma}$. Indeed, we can move the location of any discontinuity of $f$ in an interval of the form $T^j\big(x_k- \tfrac{4\sigma'}{5r_k}, x_k + \tfrac{4\sigma'}{5r_k}\big)$ to $T^j\big(x_k- \tfrac{\sigma'}{r_k}, x_k + \tfrac{\sigma'}{r_k}\big) \setminus T^j\big(x_k- \tfrac{4\sigma'}{5r_k}, x_k + \tfrac{4\sigma'}{5r_k}\big) $.

Since $\{T^i(x_k)\}_{i = 0}^{r_k - 1}$ is $\tfrac{C'}{r_k}$-dense in $[0, 1)$, it follows that $[0, 1) \setminus \bigsqcup_{j = 0}^{r_k - 1}T^j\big(x_k- \tfrac{\sigma'}{4r_k}, x_k + \tfrac{\sigma'}{4r_k}\big)$ is a finite union of closed intervals of length at most $\tfrac{C'}{r_k}.$ Hence, we can define a function $g_{i, k} \in C_{m, M}$ whose $i$-th discontinuity belongs to $ \bigsqcup_{j = 0}^{r_k - 1}T^j\big(x_k- \tfrac{\sigma'}{4r_k}, x_k + \tfrac{\sigma'}{4r_k}\big)$ and verifies $|\overline{g}_k - g_{i, k}| \leq \tfrac{C'}{r_k}\tfrac{{4M + 1}}{\Gamma}$ by moving the location of the $i$-th discontinuity of $\overline{g}_k$ by a distance of at most $\tfrac{C'}{r_k}$ (which we again assume WLOG to be smaller than $\tfrac{\Gamma}{4}$) to the closest interval in the Rokhlin tower $ \bigsqcup_{j = 0}^{r_k - 1}T^j\big(x_k- \tfrac{\sigma'}{4r_k}, x_k + \tfrac{\sigma'}{4r_k}\big).$ Moreover, we have
\[ |f - g_{i, k}| \leq |f - \overline{g}_k| + |\overline{g}_k - g_{i, k}| \leq \frac{\xi_1}{2 r_k},\]
where {$\xi_1 = \tfrac{2(m\sigma' + C)(4M + 1)}{\Gamma}$.}

Finally, setting $\xi_2 = \tfrac{\sigma'}{20}$, for any $h \in C_{m, M}$ verifying $|h - g_{i, k}| < \frac{\xi_2}{r_k}$ and any $1 \leq j \leq m$, the distance between the $j$-th discontinuities of $h$ and $g_{i, k}$ is at most $\frac{\sigma'}{20 h_{k}}.$ Hence, for any such $h$, only the $i$-th discontinuity of $h$ belongs to the set $\bigsqcup_{j = 0}^{r_k - 1}T^j\big(x_k- \tfrac{\sigma'}{2r_k}, x_k + \tfrac{\sigma'}{2r_k}\big).$

\subsection*{Step 6 - Define sequences $(f_{i, k}, x_k^{\pm}) \in K$ converging to $(f, x)$ for which $S_{r_k}f_{i, k}(x_k^{\pm})$ are bounded and differ by $\sigma'_i(f_{i, k})$} We now find two points $x^-_k$ and $x^+_k$, to the left and the right of $x_k$, respectively, such that, roughly speaking, the fiber measures corresponding to $x^+_k$ and $x^-_k$ are close to the fiber measure corresponding to $x_k$ and yet differ via the shift $V_{\sigma_i}$ (see \eqref{eq:translation}) given by the jump discontinuity. 

By \eqref{eq:quantitative_density_2}, for $k$ sufficiently large,
\[ \min\left\{\mu\big\{y \in { E_f^{B_V}} \,\big|\, |x_k - y| <\tfrac{\sigma'}{r_k}\big\}, \mu\big\{y \in { E_f^{B_V}} \,\big|\, |T^{r_k}(x_k) - y| <\tfrac{\sigma'}{r_k}\big\}\right\} > \tfrac{7}{8}\tfrac{2\sigma'}{r_k}.\]
Hence
\[ \mu\left\{\left.y \in \big(x_k - \tfrac{3\sigma'}{4r_k}, x_k - \tfrac{\sigma'}{2r_k}\big) \,\right|\, y \in { E_f^{B_V}} \text{ and } T^{r_k}(y) \in { E_f^{B_V}} \right\} > 0,\]
\[ \mu\left\{\left.y \in \big(x_k + \tfrac{\sigma'}{2r_k}, x_k + \tfrac{3\sigma'}{4r_k}\big) \,\right|\, y \in { E_f^{B_V}} \text{ and } T^{r_k}(y) \in { E_f^{B_V}} \right\} > 0.\]
Thus, there exist $x^-_k \in \big(x_k - \tfrac{3\sigma'}{4r_k}, x_k - \tfrac{\sigma'}{2r_k}\big)$ and $x^+_k \in \big(x_k + \tfrac{\sigma'}{2r_k}, x_k + \tfrac{3\sigma'}{4r_k}\big)$ such that 
\[(f, x_k^{\pm}), (f, T^{r_k}(x^{\pm}_k)) \in {B_V}. \]
Similarly, by \eqref{eq:quantitative_density_1} and for $k$ sufficiently large,
\[ \frac{\min\left\{\mu\big\{g \in P_{x_k^{\pm}}\,\big|\, |g - f| <\tfrac{\xi_1}{r_k}\big\}, \mu\big\{g \in P_{T^{r_k}(x_k^{\pm})}\,\big|\, |g - f| <\tfrac{\xi_1}{r_k}\big\} \right\}}{\mu\big\{g \in C_{m, M}\,\big|\, |g - f| <\tfrac{\xi_1}{r_k}\big\}} > 1-\tfrac{\xi_2^m}{8\sqrt{m+1}}.
\]

Hence, {noticing that $\xi_2 < \tfrac{\xi_1}{2}$ and recalling that $|f - g_{i, k}| < \tfrac{\xi_1}{2}$}, there exists $f_{i,k} \in C_{m, M}$ such that
\[ |f_{i, k} - g_{i, k}| < \tfrac{\xi_2}{r_k}, \qquad (f_{i, k}, x_k^{\pm}), (f_{i, k}, T^{r_k}(x_k^{\pm})) \in {K}.\]
Notice that by the second assertion in Step 5, we have
\[ S_{r_k}f_{i, k}(x_k^{-}) - S_{r_k}f_{i, k}(x_k^{+}) = \sigma_i(f_{i, k}).\]
Moreover, by the second assertion in Step 4 and since $|f_{i, k} - f| < \tfrac{\xi_1}{r_k}$,
\[ \left|S_{r_k}f_{i, k}(x_k^{\pm}) \right| < D + \xi_1 {+M}.\]

\subsection*{Step 7 - Use the measures $\mu_{f_{i, k}, x_k^{\pm}}$ to conclude that $(V_{\sigma_i(f)})_*(\mu_{f, x}) = \mu_{f, x}$} In the following, let $i \in \{1, \dots, m\}$ be fixed. Notice that to prove $(V_{\sigma_i(f)})_*(\mu_{f, x}) = \mu_{f, x}$, {where $V_{\sigma_i(f)}$ is as in \eqref{eq:translation},} it is sufficient to show that 
\begin{equation}
\label{eq:invariance_interval}
(V_{\sigma_i(f)})_*(\mu_{f, x})(J) = \mu_{f, x}(J),
\end{equation}
for any open bounded interval $J$. 

Recall that by definition of the set {$K$} and by the constructions of the sequences in the previous sections, we have, for any $L > 0$,
\begin{equation}
\label{eq:weak_convergence}
\mu_{f_{i, k}, x_k^{\pm}, L},\, \mu_{f_{i, k}, T^{r_k}(x_k^{\pm}), L} \xrightarrow[k \to \infty]{w} \mu_{f, x, L}.
\end{equation}

Let $J = (a, b)$ be an open bounded interval and let $L > 2\max\{|a|, |b|\} + D + \xi_1 + M +1 $. By taking a subsequence, if necessary, we may assume that $S_{r_k}f_{i, k}(x_k^{-})$ converges, as $k \to \infty$. Let us denote its limit by $t$. Denote, 
\[ J_\epsilon = (a - \epsilon, b + \epsilon),\]
for any $-\tfrac{(b - a)}{2} < \epsilon < \tfrac{(b - a)}{2}.$ Notice that for any $\epsilon$ as before, we have
\[ J_\epsilon + t \subseteq [-L, L].\]
By \eqref{eq:invariance_decomposition_BS}, for $0 < \epsilon < \tfrac{(b - a)}{2}$ sufficiently small,
\begin{align*}
\mu_{f_{i, k}, T^{r_k}(x_k^{\pm}), L}(\overline{J_{-\epsilon}} - t) & \geq \mu_{f_{i, k}, T^{r_k}(x_k^{\pm}), L}(J_{-\epsilon} - t) \\ & = \mu_{f_{i, k}, x_k^{\pm}, L}(J_{-\epsilon} - t + S_{r_k}f_{i, k}(x_k^{\pm})) \\& \geq \mu_{f_{i, k}, x_k^{\pm}, L}(J_{-2\epsilon} + \tfrac{\sigma_i(f)}{2} \pm \tfrac{\sigma_i(f)}{2}).
\end{align*}
By Portmanteau's Theorem and \eqref{eq:weak_convergence}, taking the limit as $k \to \infty$ in the previous expression yields
\[ \mu_{f, x, L}(\overline{J_{-\epsilon}} - t) \geq \mu_{f, x, L}(J_{-2\epsilon} + \tfrac{\sigma_i(f)}{2} \pm \tfrac{\sigma_i(f)}{2}). \]
By taking the limit as $\epsilon \searrow 0$, we obtain
\[ \mu_{f, x, L}(J - t) \geq \mu_{f, x, L}(J + \tfrac{\sigma_i(f)}{2} \pm \tfrac{\sigma_i(f)}{2}). \]
Similarly, by \eqref{eq:invariance_decomposition_BS} and for $0 < \epsilon < \tfrac{(b - a)}{2}$ sufficiently small,
\begin{align*}
\mu_{f_{i, k}, T^{r_k}(x_k^{\pm}), L}(J_{-2\epsilon} - t) & = \mu_{f_{i, k}, x_k^{\pm}, L}(J_{-2\epsilon} - t + S_{r_k}f_{i, k}(x_k^{\pm})) \\& \leq \mu_{f_{i, k}, x_k^{\pm}, L}(\overline{J_{-\epsilon}} + \tfrac{\sigma_i(f)}{2} \pm \tfrac{\sigma_i(f)}{2}).
\end{align*}
As before, taking limits first as $k \to \infty$ and then as $\epsilon \searrow 0$, we obtain
\[ \mu_{f, x, L}(J - t) \leq \mu_{f, x, L}(J + \tfrac{\sigma_i(f)}{2} \pm \tfrac{\sigma_i(f)}{2}). \] 
Therefore, 
\[ \mu_{f, x, L}(J - t) = \mu_{f, x, L}(J + \tfrac{\sigma_i(f)}{2} \pm \tfrac{\sigma_i(f)}{2}). \] 
In particular, 
\[ \mu_{f, x}(J + \sigma_i(f)) = \mu_{f, x, L}(J + \sigma_i(f)) = \mu_{f, x, L}(J) = \mu_{f, x}(J), \]
which shows \eqref{eq:invariance_interval}.

\subsection*{Step 8 - Conclude the proof} Since the set of jumps of $f$ generates a dense group in $\R$ and, by the previous step, we have 
\[ (V_{\sigma_i(f)})_*(\mu_{f, x}) = \mu_{f, x}, \]
for any $i = 1, \dots, m$, it follows that $\mu_{f, x}$ is a multiple of the Lebesgue measure on $\R$, which contradicts our initial assumption on the measure $\mu_{f , x}$.

\section{Proof of Theorem \ref{thm:good_returns}}
\label{sc:diophantine_condition}

Throughout this section, we fix $d, m\geq 2$, $M > 0,$ and a Rauzy class $\mathfrak{R}$. We denote by $\theta_1$ the largest Lyapunov exponent of $\mathcal{Z}$ when restricted to $\mathfrak{R} \times \Lambda^\A$. If $d > 2$, we denote by $\theta_2$ the second largest Lyapunov exponent of $\mathcal{Z}$ when restricted to $\mathfrak{R} \times \Lambda^\A$. Otherwise, we let $\theta_2 = 0$. Notice that in both cases, $0 \leq \theta_2 < \theta_1$.

First, it will be crucial for us to have control over the maximum possible growth of Birkhoff sums of a given piecewise constant cocycle over a given IET. This will be a direct consequence of the following well-known result due to Zorich \cite[Theorem 1]{zorich_deviation_1997}.

\begin{theorem}
\label{thm:deviation_classical}
For a.e. IET $T = (\pi, \lambda) \in \mathfrak{R} \times \Lambda^\A$ 
\[ \max_{x \in [0, 1)} \limsup_{n \to \infty} \tfrac{\log|\chi_\alpha(x, n) - \lambda_\alpha n|}{\log n} = \frac{\theta_2}{\theta_1},\]
where $\chi_\alpha(x, n)$ is the number of visits to $I_\alpha$ of the first $n$ iterates of $x$ by $T$. Moreover, given $\varsigma > 0$ there exists $N(T, \varsigma) \in \N$ such that for any $n \geq N$, 
\[ \max_{x \in [0, 1)} |\chi_\alpha(x, n) - \lambda_\alpha n| \leq n^{\frac{\theta_2}{\theta_1} + \varsigma}.\]
\end{theorem}

By considering IETs with $m$ marked points in the theorem above, one can easily deduce the following.

\begin{theorem}[Deviation of ergodic averages]
\label{thm:deviations}
For a.e. IET $T = (\pi, \lambda) \in \mathfrak{R} \times \Lambda^\A$ there exists a full-measure set $\mathcal{F} \subseteq C_{m, M}$ such that, for any $f \in \mathcal{F}$ and any $\varsigma > 0$, there exists $C = C(T, f, \varsigma) > 0$ satisfying
\[ \max_{x \in [0, 1)} \left| S_n^T f(x)\right| \leq C n^{\frac{\theta_2}{\theta_1} + \varsigma},\]
for any $n \geq 1$.
\end{theorem}

In addition to this, we want appropriate collections of Rokhlin towers. The main requirements to define these Rokhlin towers are the following: First, for any subtower whose height is a fixed proportion of the total height, say $\tfrac{1}{\eta}$, this subtower remains almost as dense in $[0, 1)$ (up to a constant depending on $\eta$) as the entire tower. Second, the heights of the full Rokhlin towers grow at a rate only slightly larger than $\eta^{1 + \epsilon}$, for some $\epsilon > 0$ sufficiently small.


\begin{proposition}
\label{prop:balanced_times}
Given $\epsilon > 0$, there exist $C, \eta_0, \sigma > 0$ and a positive measure set $\Delta \subseteq \mathfrak{R} \times \Lambda^\A$, depending only on $\mathfrak{R}$ and $\epsilon$, such that for a.e. IET $T = (\pi, \lambda) \in \mathfrak{R} \times \Lambda^\A$ and any $\eta > \eta_0$ there exists a sequence of recurrence times $(n_k)_{k \in \N} \subseteq \N$ (to $\Delta$ with respect to $\RVExt$) for which the sequence 
\[ h_{n_k} := \min_{\alpha \in \A} q^{(n_k)}_\alpha, \qquad k \geq 1,\]
satisfies the following:
\begin{enumerate}[i)]
\item \label{cond:towers_prop} For any $x \in I$ and any $k \geq 1$, either $\bigsqcup_{i=0}^{h_{n_k}-1}T^i\big(\big[ x - \tfrac{\sigma}{h_{n_k}}, x]\big)$ or $\bigsqcup_{i=0}^{h_{n_k}-1}T^i\big([x, x + \tfrac{\sigma}{h_{n_k}}]\big)$ consists of $h_{n_k}$ disjoint intervals.
\item \label{cond:density_prop} For any $x \in I$ and any $k \geq 1$, the set $\{T^i(x)\}_{i = 0}^{\lfloor h_{n_k}/\eta\rfloor - 1}$ is $\tfrac{C\eta^{1 + \epsilon}}{h_{n_k}}$-dense in $I$.
\item \label{cond:controlled_growth_prop} $\limsup_{k \to \infty} (h_{n_k})^{\tfrac{1}{k}} \leq C \eta^{1+\e}$.
\end{enumerate}
\end{proposition}

As we shall see in Proposition \ref{prop:recurrence}, the collections of Rokhlin towers given by Proposition \ref{prop:balanced_times} together with Theorem \ref{thm:deviations} will allow us to combine, for some point $x$ in a given positive measure subset $E$, the recurrence times of $x$ to $E$ with the times where the Birkhoff sums $S_nf(x)$ are uniformly bounded.

\begin{proof}[Proof of Proposition \ref{prop:balanced_times}] Fix $0 < \epsilon < 1$. Our construction will depend on two (small) positive constants $$0 < \nu, \delta < \min\{\epsilon, \tfrac{1}{10d}\}$$ whose exact value, which depends only on $\mathfrak{R}$ and $\epsilon$, we specify at the end on the proof (see \eqref{eq:condition_constants}). 

Let $(\pi,\hat\lambda)\in\mathfrak{R}\times\Lambda^{\mathcal A}$ be an IET, generic w.r.t. $\mathcal{R}$ such that 
\begin{itemize}
 \item $\max_{\alpha,\beta}|\hat\lambda_\alpha-\hat\lambda_{\beta}|\le\nu/2$,
 \item $1\le i<j\le d\Rightarrow \hat\lambda_{\pi_0^{-1}(i)}<\hat\lambda_{\pi_0^{-1}(j)}$.
\end{itemize}
The first condition above guarantees that the lengths of the intervals exchanged by $(\pi,\hat\lambda)$ are comparable, while the second one means that the intervals before exchange are ordered increasingly. The latter condition will allow us to show that IETs sufficiently close to $(\pi,\hat\lambda)$ satisfies \eqref{eq:split_two}. 

Let 
\[
U:= \{\pi\}\times\left\{\lambda\in\Lambda^{\A}\,\left|\, 
\max_{\alpha,\beta}|\lambda_\alpha-\lambda_{\beta}|\le \frac{\nu}{2d}\quad\text{and}\quad 1\le i<j\le d\Rightarrow \lambda_{\pi_0^{-1}(i)}< \lambda_{\pi_0^{-1}(j)}\right\}\right..
\]
Let $\ell, L \in \N$, $\gamma=\gamma_\ell(\pi,\hat \lambda)$ and $\tilde\gamma$ as in Proposition \ref{prop: pathconstruction} when applied to $(\pi,\hat\lambda)$ and $U$.

Since for any $(\pi, \lambda)\in \Delta_{\tilde\gamma}$ the intervals before exchange are ordered increasingly,
for any $\alpha \in \A\setminus\{\pi_0^{-1}(1)\}$ 
there exist {exactly} two distinct symbols $\alpha_l, \alpha_r \in \A$ such that $I_{\alpha_l}$ and $I_{\alpha_r}$ are adjacent intervals (more precisely, the right endpoint of $I_{\alpha_l}$ and the left endpoint of $I_{\alpha_r}$ coincide) satisfying
 \begin{equation}
 \label{eq:contained_two}
 T(I_\alpha) \cap I_{\alpha_l} \neq \emptyset \neq T(I_\alpha) \cap I_{\alpha_r}, \qquad T(I_\alpha) \subseteq I_{\alpha_l} \cup I_{\alpha_r}.
 \end{equation}
 Moreover,
 \begin{equation}
 \label{eq:contained_first}
 T(I_{\pi_0^{-1}(1)})\subseteq I_{\pi_1^{-1}(1)}.
 \end{equation}
 By taking a subset $\Delta\subseteq\Delta_{\tilde\gamma}$ of positive measure, and increasing the constant $C_\gamma$ if necessary, we may assume that 
\begin{itemize}
\item For any $(\pi, \lambda) \in \Delta \cup \RVKZ^{L}(\Delta) \cup \RVKZ^{2L}(\Delta)$ with exchanged intervals $\{I_\alpha\}_{\alpha \in \A}$ {
\begin{equation}
\label{eq:split_two}
T(I_\alpha) \cap I_\beta \neq \emptyset \Rightarrow |T(I_\alpha) \cap I_{\beta}| \geq C_\gamma^{-1}.
\end{equation}
}
\end{itemize}
Indeed, any IET in $\Delta_{\tilde \gamma}$ satisfies \eqref{eq:contained_two} and \eqref{eq:contained_first} and hence it also satisfies \eqref{eq:split_two} if $C_\gamma^{-1}$ is replaced by some sufficiently small constant (depending on the IET). To conclude, it suffices to notice that this constant can be taken uniformly in some small neighborhood of this IET. 


Condition \eqref{eq:split_two} will be helpful later when proving Condition \eqref{cond:towers_prop}.

Let us fix an infinitely renormalizable IET $(\pi, \lambda)$ returning to $\Delta$ infinitely often under the action of $\mathcal{Z}$ (recall that a.e. IET verifies this by the ergodicity of $\mathcal{Z}$) and denote by $(m_k)_{k \in \N}$ the associated sequence of first returns. Notice that, by Birkhoff's ergodic theorem,
\[ \lim_{k \to \infty} \frac{m_k}{k} = \frac{1}{\rho(\Delta)},\]
where $\rho$ denotes the $\mathcal{Z}$-invariant probability measure described in Section \ref{sc:notations}. Moreover, since this measure is equivalent to the Lebesgue measure in any compactly contained subset of $\mathfrak{R} \times \Lambda^\A$ (in particular, in $\mathfrak{R} \times \left\{ \lambda \in \Lambda^\A \,\left|\, |\lambda_\alpha - \lambda_\beta| \leq \tfrac{1}{10d}\right\}\right.$ which contains $\Delta$), there exists $C_d > 0$, depending only on $d$, such that $\rho(\Delta) < C_d \nu.$ Therefore,
\begin{equation}
\label{eq:frequency_return_times}
\limsup_{k \to \infty} \frac{k}{m_k} < C_d \nu.
\end{equation}
The desired sequence of return times $(n_k)_{k \in \N}$ will be a subsequence of $(m_k + 2L)_{k \in \N}$. This choice is made so that, for any $k \geq 0$ and any $x \in [0, 1)$, in addition to the length and height vectors satisfying \eqref{eq:several_balanced_lengths} and \eqref{eq:several_balanced_heights}, the orbit $\{T^i(x)\}_{i = 0}^{h_{n_k} - 1}$ goes through all the floors of the every tower in the Rokhlin tower decomposition associated to the renormalization time $m_k$. This will be useful only later in the proof of Proposition \ref{prop:recurrence}, and so, for the sake of clarity, we discuss this (and prove it) in detail in Claim \ref{cl:minimum_growth}.


By the integrability of the Zorich cocycle when induced to $\Delta$ (for additional details concerning induced cocycles, see, e.g., \cite[Section 4.4.1]{viana_lectures_2014}) and the dominated convergence theorem, we have
\begin{equation}
\label{eq:approx_integrability} \lim_{M \to \infty} \lim_{N \to \infty} \left[ \frac{1}{N} \sum_{k = 1}^{N} \log \| Q(m_{k - 1}, m_k)\| - \frac{1}{N} \sum_{k = 1}^{N} \chi _{ \| Q(m_{k - 1}, m_k)\| \leq M } \log \| Q(m_{k - 1}, m_k)\| \right] = 0,
\end{equation}
By \eqref{eq:approx_integrability}, there exists $\eta_0 \geq \exp(\tfrac{1}{\delta})$, depending only on $\Delta$ and $\delta$, such that for any $\eta \geq \eta_0$,
\begin{equation}
\label{eq:big_returns}
\frac{1}{N} \sum_{k = 1}^{N} \chi _{ \| Q(m_{k - 1}, m_k)\| > \eta^\delta } \log \| Q(m_{k - 1}, m_k)\| \leq \delta.
\end{equation}
In the following, we will always assume $\eta \geq \eta_0$.

Let $ (m_{r_k})_{k \in \N}$ be the subsequence of return times such that 
\begin{equation}
\label{eq:bounded_matrix}
\|Q(m_{r_k - 1}, m_{r_k})\| > \eta^\delta .
\end{equation}
Since for any $C > 0$ the set $\{ (\pi, \lambda) \in \mathfrak{R} \times \Lambda^\A \mid \| Q(\pi, \lambda)\| > C \}$ has positive measure (see, e.g., ), by ergodicity of the Kontsevich-Zorich cocycle this sequence is infinite for a typical IET. Moreover, this subsequence satisfies
\begin{equation}
\label{eq:frequency_subsequence}
\limsup_{k \to \infty} \frac{k}{r_k} \leq \frac{\delta}{\log \eta^\delta } < \min \left\{ \delta, \frac{1}{\log \eta}\right\}.
\end{equation}
{
Indeed, by \eqref{eq:big_returns},
\[ \frac{k \log(\eta^\delta) }{r_k} \leq \frac{1}{r_k} \sum_{i = 1}^{r_k} \chi _{ \| Q(m_{i - 1}, m_i)\| > \eta^\delta } \log \| Q(m_{i - 1}, m_i)\| \leq \delta. \]
}
Denote
\[ C_\Delta := dC_\gamma^2 \|A_\gamma\|^3\]
and define $(n_k)_{k \in \N} := (m_{l_k} + 2L)_{k \in \N}$ inductively by setting $l_0 = 0$ and, for any $k \geq 0$,
\begin{equation}
\label{eq:matrix_norm_growth}
l_{k + 1}:= \min \left\{ l > l_k \,\left|\, \exists l_k < j < l \,\text{ s.t. }\, \begin{array}{l} 
\|Q(m_j, m_l)\| \geq \eta C_\Delta, \\
\|Q(m_i, m_{i + 1})\| \leq \eta^\delta , \text{ for } j \leq i < l. \end{array} \right\}\right..
\end{equation}
Let us point out that the condition above gives rise to a well-defined infinite sequence $(l_k)_{k \in \N}$. Indeed, if only $K \geq 0$ elements of this sequence can be defined, for any $N$ such that $r_N > l_K$ and any $M > N$, by \eqref{eq:big_returns}, 
\begin{align*}
\|Q(m_{r_N}, m_{r_M})\|& \leq \prod_{i = N + 1}^M \|Q(m_{r_i - 1}, m_{r_i})\| \prod_{i = N}^M \|Q(m_{r_i}, m_{r_{i + 1} - 1})\| \\
& < \exp(\delta r_M) (\eta C_\Delta)^{M - N}.
\end{align*}
Hence, by \eqref{eq:frequency_subsequence}, if $M$ is sufficiently large then
\begin{align*}
\frac{1}{r_M}\log \|Q(m_{r_N}, m_{r_M})\| &\leq \delta + \frac{M}{r_M}\left( \log \eta + \log C_\Delta\right) \\
& < \delta + 1 + \delta \log C_\Delta.
\end{align*}
Assuming that $\delta$ is sufficiently small so that the RHS in the previous equation is smaller than $\log 2$, this contradicts the fact that $ \|Q(m_{r_N}, m_{r_M})\| \geq \|A_\gamma\|^{r_M - r_N} \geq 2^{r_M - r_N}$, for any $M > N \geq 1$, where we use the fact that $A_\gamma$ is a positive matrix.

Therefore, using \eqref{eq:matrix_norm_growth}, we can define sequences $(l_k)_{k \in \N}, (s_k)_{k \in \N} \subseteq \N$ such that
\[ \eta C_\Delta \leq \|Q(m_{l_k - s_k}, m_{l_k}) \| \leq \|Q(m_{l_k - s_k}, m_{l_k - 1}) \| \|Q(m_{l_k - 1}, m_{l_k}) \| \leq \eta^{1 + \delta} C_\Delta, \qquad l_k < l_{k + 1} - s_{k + 1},\] for any $k \geq 0$. 

Denote $(\overline{n}_k)_{k \in \N} := (m_{l_k - s_k})_{k \in \N}$. The equation above implies that
\begin{equation}
\label{eq:sequence_norm_growth}
\eta C_\Delta \|A_\gamma\|^{-1} \leq \|Q(\overline{n}_k + L, n_k - L)\| = \|Q(m_{l_k - s_k} + L, m_{l_k} + L)\| \leq \eta^{1 + \delta} C_\Delta \|A_\gamma\|.
\end{equation}

We now check that this sequence satisfies all the properties in the statement of the proposition. \\



\underline{Condition \eqref{cond:towers_prop}}: Define $\sigma := \tfrac{1}{10dC_\gamma}$. Fix $x \in [0, 1)$ and $k \geq 1.$ Let $\alpha \in \A$ and $0 \leq i < q^{(n_k)}_\alpha$ such that $x \in T^i\big(I^{(n_k)}_\alpha\big).$ {Let $\alpha_l, \alpha_r \in \A$ be two symbols such that $T^{q^{(n_k)}_\alpha}(I^{(n_k)}_\alpha) \subseteq I^{(n_k)}_{\alpha_l} \cup I^{(n_k)}_{\alpha_r}$ and $I^{(n_k)}_{\alpha_l}, I^{(n_k)}_{\alpha_r}$ are adjacent intervals (see equations \eqref{eq:contained_two} and \eqref{eq:contained_first}). By \eqref{eq:split_two},} either 
\[ \big[x - \tfrac{\sigma}{q^{(n_k)}_\alpha}, x] \subseteq T^i\big(I^{(n_k)}_\alpha\big), \qquad \text{and} \qquad T^{q^{(n_k)}_\alpha - i}\big(\big[x - \tfrac{\sigma}{q^{(n_k)}_\alpha}, x]\big) \subseteq I^{(n_k)}_{\alpha_l} \text{ or } T^{q^{(n_k)}_\alpha - i}\big(\big[x - \tfrac{\sigma}{q^{(n_k)}_\alpha}, x]\big) \subseteq I^{(n_k)}_{\alpha_r},\]
or 
\[ \big[x, x + \tfrac{\sigma}{q^{(n_k)}_\alpha}] \subseteq T^i\big(I^{(n_k)}_\alpha\big), \qquad \text{and} \qquad T^{q^{(n_k)}_\alpha - i}\big(\big[x, x + \tfrac{\sigma}{q^{(n_k)}_\alpha}]\big) \subseteq I^{(n_k)}_{\alpha_l} \text{ or } T^{q^{(n_k)}_\alpha - i}\big(\big[x, x + \tfrac{\sigma}{q^{(n_k)}_\alpha}]\big) \subseteq I^{(n_k)}_{\alpha_r}.\]
Let us point out that both sets of conditions can hold simultaneously. 

Assume WLOG that 
\[\big[x, x + \tfrac{\sigma}{q^{(n_k)}_\alpha}] \subseteq T^i\big(I^{(n_k)}_\alpha\big)\quad \text{and} \quad T^{q^{(n_k)}_\alpha - i}\big(\big[x, x + \tfrac{\sigma}{q^{(n_k)}_\alpha}]\big) \subseteq I^{(n_k)}_{\alpha_l},\] the other cases being analogous.

Then, since $h_{n_k} = \min_{\alpha} q^{(n_k)}_\alpha$, it follows that $\bigcup_{i = 0}^{h_{n_k} - 1 }T^i\big(\big[x, x + \tfrac{\sigma}{q^{(n_k)}_\alpha}]\big)$ is a union of disjoint intervals contained in the union of the Rokhlin towers $\bigcup_{i = 0}^{q^{(n_k)}_\alpha - 1 }T^i\big(I^{(n_k)}_\alpha\big)$ and $\bigcup_{i = 0}^{q^{(n_k)}_{\alpha_l} - 1 }T^i\big(I^{(n_k)}_{\alpha_l}\big)$ associated to the $n_k$-th step of renormalization. 

\underline{Condition \eqref{cond:density_prop}}: Notice that for any $x \in [0, 1)$, any $k \geq 0$ and any $\alpha \in \A$, the orbit $\{T^i(x)\}_{i = 0}^{q^{(m_k + 2L)}_\alpha - 1}$ is $\max_{\beta} \lambda_\beta^{(m_k + L)}$-dense in $[0, 1)$. 

Indeed, since the first $L$ Zorich renormalizations of $(\pi^{(m_k + L)}, \lambda^{(m_k + L)})$ follow the path $\gamma$ and $A_\gamma$ is a positive matrix, the orbit $\{T^i(x)\}_{i = 0}^{q^{(m_k + 2L)}_\alpha - 1}$ goes through all Rokhlin towers associated to $(\pi^{(m_k + L)}, \lambda^{(m_k + L)})$ at least once. 

 Fix $x \in [0, 1)$ and $k \geq 1.$ We have
\begin{align*}
h_{n_k} & \geq \tfrac{1}{dC_\gamma} \big|q^{(n_k)}\big| = \tfrac{1}{dC_\gamma} \big|Q(\overline{n}_k, n_k)q^{(\overline{n}_k)}\big| \\
& = \tfrac{1}{dC_\gamma} \big|A_\gamma Q(\overline{n}_k + L, n_k - L)A_\gamma q^{(\overline{n}_k)} \big| \\
& \geq \tfrac{1}{dC_\gamma} \big\| Q(\overline{n}_k + L, n_k - L)\big\| \big| q^{(\overline{n}_k)} \big| \\
& \geq \eta \big| q^{(\overline{n}_k)} \big|,
\end{align*}
where, in the third inequality, we use the fact that the matrix $A_\gamma$ is positive and the matrix $Q(\overline{n}_k + L, n_k - L)$ is non-negative, and, in the fourth inequality, we apply \eqref{eq:sequence_norm_growth}.

Hence, the orbit $\{T^i(x)\}_{i = 0}^{\lfloor h_{n_k}/\eta\rfloor - 1}$ goes fully along at least one of the Rokhlin towers associated with $(\pi^{(\overline{n}_k)}, \lambda^{(\overline{n}_k)})$. 
It follows from the remark above that this orbit is $\max_{\beta} \lambda_\beta^{(\overline{n}_k - L)}$-dense in $[0, 1)$. 

Since $\sum_{\beta \in \A} \lambda_\beta^{(\overline{n}_k - L)} q_\beta^{(\overline{n}_k - L)} = 1$, by \eqref{eq:several_balanced_lengths} and \eqref{eq:several_balanced_heights}, 
\[\max_{\beta} \lambda_\beta^{(\overline{n}_k - L)} \leq \tfrac{C_\gamma}{d \big|q^{(\overline{n}_k - L)}\big|}.\]
Hence, noticing that
\[ h_{n_k} \leq C_\gamma |q^{(n_k)}|, \quad q^{(n_k)} = Q(\overline{n}_k - L, n_k) q^{(\overline{n}_k - L)}, \quad \| Q(\overline{n}_k - L, n_k)\| \leq \eta^{1 + \delta}C_\Delta \|A_\gamma\|^2,\]
it follows that 
\[
\max_{\beta} \lambda_\beta^{(\overline{n}_k - L)} \leq \tfrac{C_\gamma^2 \eta^{1 + \delta} C_\Delta \|A_\gamma\|^2}{d h_{n_k}}.
\]
Therefore, denoting 
\[C := \frac{C_\gamma^2 C_\Delta \|A_\gamma\|^2}{d},\]
and since $ 0 < \delta < \epsilon$, it follows that the orbit $\{T^i(x)\}_{i = 0}^{\lfloor h_{n_k}/\eta\rfloor - 1}$ is $\tfrac{C\eta^{1 + \epsilon}}{h_{n_k}}$-dense in $[0, 1)$. \\

\underline{Condition \eqref{cond:controlled_growth_prop}:} Fix $k \geq 1$. Notice that
\[\|Q(n_{k - 1}, n_k)\| \leq (\eta^{1 + \delta}C_\Delta \|A_\gamma\|^2) (\eta C_\Delta)^{p_k} \prod_{\substack{i \in \N \\ n_{k - 1} < m_{r_i} \leq n_k}} \|Q(m_{r_i - 1}, m_{r_i})\|,\]
where $p_k:= \# \{i \in \N \mid n_{k - 1} < m_{r_i} \leq n_k \}.$ Indeed, { recalling that $n_k = m_{l_k} + 2L$ (see \eqref{eq:matrix_norm_growth}),} the times $m_{r_i}$ (see \eqref{eq:bounded_matrix}) satisfying $n_{k - 1} < m_{r_i} \leq n_k$ divide the return times between $m_{l_{k - 1}} = n_{k - 1} - 2L$ and $m_{l_k} = n_k - 2L$ into at most $p_k + 1$ blocks (which do not contain any $m_{r_i}$) so that, on each block, the return times satisfy $\|Q(m_i, m_{i + 1})\| \leq \eta^\delta$. By \eqref{eq:matrix_norm_growth}, the norm of the first $p_k$ blocks is smaller or equal to $\eta C_\Delta$ while the norm of the last block is smaller or equal to $\eta^{1 + \delta} C_\delta$. More precisely, if we denote these $p_k$ return times by
\[n_{k - 1} < m_{t_1} < m_{t_2} < \dots m_{t_{p_k}} < \overline{n}_k < n_k,\]
then, 
\[\left\{ \begin{array}{l} 
\|Q(m_{{t_i} - 1}, m_{t_i})\| \geq \eta^\delta, \qquad\quad \text{ for } i = 1, \dots, p_k - 1, \\
\|Q(m_{t_i}, m_{t_{i + 1} - 1})\| \leq \eta C_\Delta, \qquad \text{ for } i = 1, \dots, p_k - 1, \\
\|Q(n_{k - 1}, m_{t_1 - 1})\| \leq \|m_{l_{k - 1}}, m_{t_1 - 1})\| \leq \eta C_\Delta,\\
\|Q(m_{t_{p_k}}, n_k)\| \leq \| Q(m_{t_{p_k}}, m_{l_k - 1})\| \| Q(m_{l_k - 1}, m_{l_k})\| \| Q(m_{l_k}, m_{l_k + 2L})\| \leq \eta^{1 + \delta}C_\Delta \|A_\gamma\|^2.
\end{array} \right. \]

%

Hence, 
\begin{equation}
\label{eq:initial_exp_bound}
\|Q(n_0, n_k)\| \leq (\eta^{1 + \delta}C_\Delta \|A_\gamma\|^2)^{k} (\eta C_\Delta)^{P_k} \prod_{\substack{m_0 < m_{r_i} \leq m_{l_k}}} \|Q(m_{r_i - 1}, m_{r_i})\|,
\end{equation}
where $P_k:= \sum_{j = 1}^k p_j.$ Notice that by \eqref{eq:big_returns}, $P_k \leq \frac{l_k}{\log \eta}$. 

By \eqref{eq:frequency_return_times}, for $k$ sufficiently large, $\frac{l_k}{n_k} < C_d \nu$. Hence, using \eqref{eq:big_returns} to bound the product along the return times $m_{r_i}$ in \eqref{eq:initial_exp_bound}, for $k$ sufficiently large, we obtain
\begin{align*}
\frac{1}{k} \log \|Q(n_0, n_k)\| & \leq \log(\eta^{1 + \delta} C_\Delta\|A_\gamma\|^2) + \frac{l_k}{k\log \eta} \log(\eta C_\Delta) + \delta \frac{l_k}{k} \\
& \leq \log(\eta^{1 + \delta} C_\Delta\|A_\gamma\|^2) + \frac{l_k}{k} \left( 1 + \frac{\log(C_\Delta)}{\log\eta} + \delta\right) \\
& \leq \log(\eta^{1 + \delta} C_\Delta\|A_\gamma\|^2) + \frac{n_k}{k} C_d\nu (1 + 2\delta),
\end{align*}
assuming $\eta$ is sufficiently large so that $ \frac{\log(C_\Delta)}{\log\eta} < \delta.$

Since $\lim_{k \to \infty} \frac{\log \|Q(0, N)\|}{N} = \theta_1$ 
then, for $k$ sufficiently large, the previous equation yields,
\begin{align*}
\frac{n_k}{k} & = \frac{n_k}{\log \|Q(n_0, n_k)\|} \frac{\log \| Q(n_0, n_k) \|}{k} \\
& \leq \frac{1}{\theta_1 - \delta} \left( \log(\eta^{1 + \delta} C_\Delta\|A_\gamma\|^2) + \frac{n_k}{k} C_d\nu (1 + 2\delta)
 \right).
\end{align*}
Thus
\[ \frac{n_k}{k} \leq \frac{\log(\eta^{1 + \delta} C_\Delta\|A_\gamma\|^2)}{\theta_1 - \delta} \left(1 - \frac{C_d\nu(1 + 2\delta)}{\theta_1 - \delta} \right)^{-1}.\]
Finally, for $k$ sufficiently large,
\begin{align*}
(h_{n_k})^{1/k} & \leq | Q(0, n_k)[1,\dots,1] |^{1/k} \\
& \leq \exp\left(( \theta_1 + \delta)\frac{n_k}{k} \right) \\
& \leq (\eta^{1 + \delta} C_\Delta\|A_\gamma\|^2)^{\tfrac{\theta_1 + \delta}{\theta_1 - 2\delta}\left(1 - \frac{C_d\nu(1 + 2\delta)}{\theta_1 - \delta} \right)^{-1}}\\
& \leq (\eta C_\Delta\|A_\gamma\|^2)^{1 + \epsilon},
\end{align*}
where we assume that $\nu$ and $\delta$ are sufficiently small, so that
\begin{equation}
\label{eq:condition_constants}
(1 + \delta) \frac{\theta_1 + \delta}{\theta_1 - 2\delta}\left(1 - \frac{C_d\nu(1 + 2\delta)}{\theta_1 - \delta} \right)^{-1} \leq 1 + \epsilon,
\end{equation}
where $C_d$ is the constant in \eqref{eq:frequency_return_times}, which depends only on $d$.
\end{proof}
 Using the previous `balanced times' given by Proposition \ref{prop:balanced_times}, we can prove the following.

\begin{proposition}
\label{prop:recurrence}
Fix $\epsilon > 0$ satisfying
\begin{equation}
\label{eq:smallness_cond_epsilon}
(1 + \epsilon)\frac{\theta_2}{\theta_1} < 1,
\end{equation}
and let $\Delta, \gamma, L, C, C_\gamma, \nu, \eta_0, \sigma$ as in Proposition \ref{prop:balanced_times}.

There exists $\eta \geq \eta_0$, depending only on $\mathfrak{R}$ and $\epsilon$, such that for any IET $T = (\pi, \lambda) \in \mathfrak{R} \times \Lambda^\A$ as in Proposition \ref{prop:balanced_times} and Theorem \ref{thm:deviations}, and for a.e $f \in C_{m, M}$, the following holds.

Let $(n_k)_{k \geq 1}$, $(h_{n_k})_{k \geq 1} \subseteq \N$ be the sequences given by Proposition \ref{prop:balanced_times} associated to $T$ and $\eta$. Let $D > mM$ and $E \subseteq [0, 1)$ with positive Lebesgue measure. Then, for any $P \in \N$ there exist $y \in E$, and natural numbers $p \geq P$ and $\tfrac{h_{n_p}}{\eta} \leq n \leq h_{n_p}$, such that: 
\begin{enumerate}
\item $T^n(y) \in E$,
\item $|S_n f(y)| < D$.
\end{enumerate}
\end{proposition}

\begin{proof}[Proof of Proposition \ref{prop:recurrence}]
Fix $\varsigma > 0$ satisfying 
\begin{equation}
\label{eq:smallness_cond_epsilon_varsigma}
(1 + \epsilon)\left(\frac{\theta_2}{\theta_1} + \varsigma\right) < 1 - \varsigma.
\end{equation}
Notice that such a value exists since, by assumption, $\epsilon$ verifies \eqref{eq:smallness_cond_epsilon}. Fix $\eta \in 2\N$ satisfying
\begin{equation}
\label{eq:large_eta}
\eta > \max\left\{\eta_0, (4C)^{1/\varsigma}\right\},
\end{equation}
and let $(n_k)_{k \geq 1}$, $(h_{n_k})_{k \geq 1} \subseteq \N$ be the sequences given by Proposition \ref{prop:balanced_times} associated to $T$ and $\eta$.

Assume, for the sake of contradiction, that the conclusions of Proposition \ref{prop:recurrence} do not hold for some $D, P > 0$ and some $E \subseteq [0, 1)$ with $\overline{\delta} := |E| > 0$. 

Up to considering a positive measure subset of $E$, we may assume, without loss of generality,
\begin{equation}
\label{eq:minimum_return_time}
\inf_{x \in E} \left( \inf \{i \geq 1 \mid T^i(x) \in E\}\right) \geq h_{n_P}.
\end{equation}

Fix $\vartheta > 0$ satisfying
\begin{equation}
\label{eq:vartheta_condition}
\vartheta < \tfrac{1}{8\sigma \|A_\gamma\|^2 C_\gamma (1 + \nu)},
\end{equation}
and let $x \in E$ be a density point of $E$. Fix $k > P$ and define
\[ L_0(k) := \left\{ 0 \leq i < h_{n_k} \, \left| \, |F_{k}(T^i(x)) \cap E | > (1 - \vartheta)|F_{k}(T^i(x))| \right. \right\}, \]
where $F_{k}(T^i(x))$ denotes the unique floor containing $T^i(x)$ in the Rokhlin towers decomposition associated to the renormalization time $n_k - 2L$. 

{
\begin{claim}
\label{cl:minimum_growth}
There exists $k_0 > 0$ and $0 < \delta < \overline{\delta}$ such that $|L_0(k)| \geq \delta h_{n_k}$, for any $k \geq k_0$.
\end{claim}
\begin{proof}[Proof of Claim \ref{cl:minimum_growth}]
Notice that by the choice of $\Delta$ in Proposition \ref{prop:balanced_times}, the orbit $\{x, \dots, T^{h_{n_k} - 1}(x)\}$ intersects all floors in the Rokhlin towers decomposition associated to the renormalization time $n_k - 2L$. In particular,
 \begin{equation}
 \label{eq:full_tower}
 [0, 1) = \bigcup_{i = 0}^{h_{n_k} - 1} F_k(T^i(x)).
 \end{equation}
 Indeed, as $q^{(n_k)} = A_\gamma q^{(n_k - L)}$ and all the entries of $A_\gamma$ are larger or equal to $2$, it follows that $h_{n_k} \geq 2 |q^{(n_k - L)}|$. In particular, the orbit $\{x, \dots, T^{h_{n_k} - 1}(x)\}$ goes through at least one full Rokhlin tower associated with the renormalization time $n_k - L$. Moreover, since $q^{(n_k - L)} = A_\gamma q^{(n_k - 2L)}$ and $A_\gamma$ is positive, this orbit goes through all the Rokhlin towers associated with the renormalization time $n_k - 2L$ at least once. 

Since a.e. $x \in E$ is a density point of $E$, there exists $t_0 > 0$ such that the set
\[ \overline{E} := \left\{ x \in E \,\left| \, \forall 0 < t < t_0, \quad \frac{\textup{Leb}_{[0, 1)}\{y\in E \mid |x - y| < t\}}{2t}> 1 - \tfrac{\vartheta}{2} \right\} \right. \]
satisfies 
\[|\overline{E}| > \tfrac{\overline{\delta}}{2}.\]
Let $k_0 > 0$ sufficiently large so that, for any $k \geq k_0$, 
\[ \max_{\alpha \in \A} \lambda_\alpha^{(n_k - 2L)} < t_0. \]
Fix $k \geq k_0$ and let us denote 
\[ \overline{L} = \{ 0 \leq i < h_{n_k} \mid \overline{E} \cap F_k(T^i(x)) \neq \emptyset \}. \]
Notice that, by definition of $\overline{E}$,
\[\overline{L} \subseteq L_0(k).\]
By \eqref{eq:full_tower} we have $\overline{E} \subseteq \bigcup_{i = 0}^{h_{n_k} - 1} F_k(T^i(x))$. Thus, 
\[\frac{\overline{\delta}}{2} < |\overline{E}| \leq \bigcup_{i \in \overline{L}} F_k(T^i(x)) \leq |\overline{L}|\max_{\alpha \in \A} \lambda_{\alpha}^{(n_k - 2L)} \leq |L_0(k)|\max_{\alpha \in \A} \lambda_{\alpha}^{(n_k)} \leq |L_0(k)| \frac{d(1 + \nu) C_\gamma}{h_{n_k}},\]
where in the last inequality we use $\sum_{\alpha \in \A} \lambda_\alpha^{(n_k)} q_\alpha^{(n_k)} = 1$, together with \eqref{eq:several_balanced_lengths} and \eqref{eq:several_balanced_heights}.

The result holds by setting $\delta := \overline{\delta}\frac{1}{2d(1 + \nu)C_\gamma}$. 
\end{proof}
In the following, we assume that $k \geq \max\{P, k_0\}$, where $k_0$ is given by the previous claim. 

}
\begin{claim}
There exists $L(k) \subseteq L_0(k)$ such that $|L(k)| \geq \delta \left( \tfrac{\eta}{4}\right)^{k - P}$ and, for any $a, b \in L(k)$, if $a \neq b$ then there exists $P \leq p \leq k$ such that $\tfrac{h_{n_p}}{\eta} \leq |a - b| \leq h_{n_p}$.
\end{claim}

\begin{proof}[Proof of the Claim]
Starting from $L_0(k)$, we define a nested sequence of sets 
$$L(k):= L_{k - P}(k) \subseteq L_{k - P - 1}(k) \subseteq \dots \subseteq L_0(k) \subseteq \{0, \dots, h_{n_k}\},$$ as follows. First, we split $ \{0, \dots, h_{n_k}\}$ in $\eta$ disjoint pieces with at {most} $ \left\lfloor \tfrac{h_{n_k}}{\eta} \right\rfloor$ consecutive elements by setting
\[J_{0, i}(k) = \left\{i \left\lfloor \tfrac{h_{n_k}}{\eta} \right\rfloor, \dots, i \left\lfloor \tfrac{h_{n_k}}{\eta}\right\rfloor - 1 \right\}, \qquad i = 0, \dots \eta - 1, \]
and 
\[J_{0, \eta}(k) = \left\{\eta \left\lfloor \tfrac{h_{n_k}}{\eta} \right\rfloor, \dots, h_{n_k} \right\}.\]
Denote $L_{0, i}(k):= L_0(k) \cap J_{0, i}(k)$, for $i = 0, \dots, \eta$. It follows {from Claim \ref{cl:minimum_growth}} that 
\begin{equation}
\label{eq:even_odd_split}
 \left | \bigsqcup_{i \in 2\N + \varphi(k)} L_{0, i}(k) \right| = \max\left\{ \left | \bigsqcup_{i \textup{ odd}} L_{0, i}(k) \right|, \left | \bigsqcup_{i \textup{ even}} L_{0, i}(k) \right| \right\} \geq \frac{|L_0(k)|}{2} \geq \delta \frac{h_{n_k}}{2},
\end{equation}
where $\varphi(k) = 1$ if the maximum in the equation above is attained for $i$ odd, and $\varphi(k) = 0$ otherwise. 

Since by assumption $h_{n_{k - 1}} \leq \left\lfloor \tfrac{h_{n_k}}{\eta} \right\rfloor$, for each $1 \leq i \leq \eta$ there exists a subset $\overline{J}_{0, i} \subseteq J_{0, i}(k)$ of $h_{n_{k - 1}}$ consecutive elements such that $\overline{L}_{0, i}(k):= L_0(k) \cap \overline{J}_{0, i}(k)$ satisfies
\[ |\overline{L}_{0, i}(k)| \geq \frac{ \eta h_{n_{k - 1}}}{2h_{n_k}}|L_0(k) \cap J_{0, i}(k)|.\]
We define 
\[ L_1(k) := \bigsqcup_{i \in 2\N + \varphi(k)} \overline{L}_{0, i}(k),\]
which by \eqref{eq:even_odd_split} verifies
\[ \left | L_1(k) \right| \geq \frac{ \eta h_{n_{k - 1}}}{2h_{n_k}} \frac{|L_0(k)|}{2} \geq \delta h_{n_{k - 1}} \frac{\eta }{4}.\]
Notice that by construction, $L_1(k)$ is the disjoint union of $\tfrac{\eta}{2}$ disjoint sets, and each of these sets is contained in a subset of $h_{n_{k - 1}}$ consecutive natural numbers. Moreover, if $a$ and $b$ belong to different pieces of this union, then $|b - a| > h_{n_k}$. 

Repeating the process described above to each of these sets in the disjoint union, we can define, recursively, sets $L_{k - P}(k) \subseteq L_{k - P - 1}(k) \subseteq \dots \subseteq L_0(k) \subseteq \{0, \dots, h_{n_k}\},$ such that, for any $0 \leq i \leq k - P$, the set $L_i(k)$ verifies $|L_{i}(k)| \geq \delta h_{n_{k - i}} \left(\frac{\eta}{4}\right)^i$ and it is the disjoint union of (possibly empty) $ \left(\frac{\eta}{2}\right)^i$ sets, each of these contained in a subset of $h_{n_{k - i}}$ consecutive natural numbers. Moreover, if $a$ and $b$ belong to different pieces of this union, there exists $k - i \leq p \leq k$ such that $\tfrac{h_{n_p}}{\eta} \leq |a - b| \leq h_{n_p}$.

Finally, the claim follows by noticing that by \eqref{eq:minimum_return_time} the $\left(\frac{\eta}{2}\right)^{k - P}$ disjoint sets defining $L(k)= L_{k - P}(k)$ consist of at most one element. Indeed, if this was not the case, there exist $a, b \in L(k) \subseteq L_0(k)$ such that $0 < b - a < h_{n_p}.$ In particular, 
\begin{equation}
\label{eq:set_intersection}
|F_k(T^a(x)) \cap E| > (1 - \vartheta)|F_k(T^a(x))|, \qquad |F_k(T^b(x)) \cap E| > (1 - \vartheta)|F_k(T^b(x))|.
\end{equation}

If $T^a(x)$ and $T^b(x)$ belong to the same tower in the decomposition associated with the renormalization time $n_k - 2L$, then, by \eqref{eq:vartheta_condition} and \eqref{eq:set_intersection}, $T^{b - a}(E) \cap E \neq \emptyset$, which contradicts \eqref{eq:minimum_return_time}. 

Otherwise, if $T^a(x)$ and $T^b(x)$ belong to a different tower in this decomposition, by \eqref{eq:several_balanced_lengths} and \eqref{eq:split_two}, 
\[|T^{b - a}(F_k(T^a(x)) \cap F_k(T^b(x))| > \frac{1}{C_\gamma (1 + \nu)} |F_k(T^b(x))|,\]
and thus, by \eqref{eq:vartheta_condition} and \eqref{eq:set_intersection}, $T^{b - a}(E) \cap E \neq \emptyset$ which again contradicts \eqref{eq:minimum_return_time}. 
\end{proof}

Recall that either $\big[ x - \tfrac{\sigma}{h_{n_k}}, x]$ or $\big[x, x + \tfrac{\sigma}{h_{n_k}}\big]$ is a continuity interval of $T^{h_{n_k}}$. Let us denote this continuity interval by $J_x$. Notice that, since $f \in C_{m, M}$ is piecewise constant {and $\bigsqcup_{i = 0}^{h_{n_k} - 1} T^i(J_x)$ is a disjoint union of intervals,} we have
{
\begin{equation}\label{eq: valuepropagates}
 |S_if(y) - S_if(y')| < mM, \qquad \text{for any } y, y' \in J_x \text{ and }0 \leq i < h_{n_k}.
\end{equation}
Moreover, for any $i, j \in L(k)$ with $i < j$, there exists $y \in T^i(J_x) \cap E$ such that $T^{j - i}(y) \in E.$ 
Indeed, {by the definition of $L(k)$}, we have
\begin{align*}
\big|T^{j - i}\big(T^i(J_x \cap E)\big) \cap E\big| & = |T^i(J_x\cap E)|-\big|T^{j - i}\big(T^i(J_x \cap E)\big) \setminus E\big|\\
& \ge |T^i(J_x\cap E)|-\big|F_k(T^j(x)) \setminus E\big|\\
& \ge|T^i(J_x \cap E)| - \vartheta |F_k(T^j(x))| \\
& \geq |T^i(J_x)| - \vartheta |F_k(T^i(x))|- \vartheta |F_k(T^j(x))|\\
& \geq \frac{\sigma}{h_{n_k}}- 2 \vartheta \max_{\beta \in \A} \lambda_\beta^{(n_k - 2L)}.
\end{align*}
Since $\sum_{\beta \in \A} \lambda_\beta^{(n_k - 2L)} q_\beta^{(n_k - 2L)} = 1$ and $q^{(n_k)} = A_\gamma^2 q^{(n_k - 2L)}$, by \eqref{eq:several_balanced_lengths}, 
\[
\max_{\beta} \lambda_\beta^{(n_k - 2L)} \leq \tfrac{(1 + \nu)\|A_\gamma\|^2}{h_{n_k}}.
\]
Hence, by \eqref{eq:vartheta_condition}, 
\[\big|T^{j - i}\big(T^i(J_x \cap E)\big) \cap E\big| \geq \frac{\sigma - 2\vartheta (1 + \nu)\|A_\gamma\|^2}{h_{n_k}} \geq \frac{3\sigma}{4h_{n_k}}.\]
and thus there exists $y \in T^i(J_x) \cap E \cap T^{i - j}(E).$ 

Therefore, if we assume that the conclusions of Proposition \ref{prop:recurrence} are false, it follows that $$|S_{j - i}f(y)| > D,$$ for any $i, j \in L(k)$ with $i < j$ and any $y \in T^i(J_x)$ {$\cap E \cap T^{i - j}(E)$}. {In view of \eqref{eq: valuepropagates}, this implies that $$|S_{j - i}f(y)| > D - mM,$$
for any $i, j \in L(k)$ with $i < j$ and any $y \in T^i(J_x)$.} In particular, letting $i_k := \min(L(k))$, 
\[|S_{j - i_k} f (T^{i_k}(z)) - S_{i - i_k} f(T^{i_k}(z))| = |S_{j - i}f(T^{i}(z))| > D - mM.\]
for any $i, j \in L(k)$ with $i_k < i < j$ and any $z \in J_x$. 

Denote $D' := D - mM$}. The previous equation implies
\[ \left| \bigsqcup_{i \in L(k) \setminus \{i_k\}} \left(S_{i - i_k}f(T^{i_k}(x)) - \tfrac{D'}{2}, S_{i - i_k}f(T^{i_k}(x)) + \tfrac{D'}{2}\right) \right| \geq D'(|L(k)| - 1).\]
In particular, there exists $i \in L \setminus \{i_k\}$ such that 
\[ |S_{i - i_k}f(T^{i_k}(x))| \geq \frac{D'}{2}(|L(k)| - 1) \geq \frac{D'}{2} \left( \delta \left( \tfrac{\eta}{4}\right)^{k - P} - 1\right).\]
Hence
\[ \max\left\{ |S_{i_k}f(x)|, |S_{i}f(x)| \right\} \geq \frac{D'}{4} \left( \delta \left( \tfrac{\eta}{4}\right)^{k - P} - 1\right).\] 
On the other hand, it follows from Theorem \ref{thm:deviations} that
\[ \max\left\{ |S_{i_k}f(x)|, |S_{i}f(x)| \right\} \leq C_0h_{n_k}^{\frac{\theta_2}{\theta_1} + \varsigma}.\]
where $C_0$ is a positive constant depending only on $T, f, \varsigma$. The last two equations imply that 
\[ C_0h_{n_k}^{\frac{\theta_2}{\theta_1} + \varsigma} \geq \frac{D'}{2} \left( \delta \left( \tfrac{\eta}{4}\right)^{k - P} - 1\right).\]
Since the construction above can be done for any $k > P$, taking $k$-th root and making $k$ go to infinity, Proposition \ref{prop:balanced_times} yields
\[ (C\eta)^{(1 + \epsilon)\left(\frac{\theta_2}{\theta_1} + \varsigma\right)} \geq \frac{\eta}{4}.\]
Thus, by \eqref{eq:smallness_cond_epsilon_varsigma}, 
\[ C\eta^{1 - \varsigma} \geq \frac{\eta}{4},\]
which contradicts \eqref{eq:large_eta}.
\end{proof}

Theorem \ref{thm:good_returns} now follows quickly from Propositions \ref{prop:balanced_times} and \ref{prop:recurrence}.

\appendix
\section{}
\label{appendix}

This section gives a detailed proof of the first step in the proof of Theorem \ref{thm:main_thm}, namely, the following.

\begin{proposition}
\label{prop:measurable_dependence_measures}
Let $M > 0$, $m \in \N$ and $d \geq 2$ be fixed. Given an IET $T$ on $d$ intervals ergodic w.r.t. the Lebesgue measure on $[0, 1)$ and a positive measure set $W \subseteq C_{m, M}$ such that for any $f \in W$ the associated skew product $T_f$ given by \eqref{eq:skew_product} is not ergodic w.r.t. to the product of Lebesgue measures on $[0, 1) \times \R$, there exists a positive measure subset $V \subseteq W$ and a measurable function $f \mapsto \mu_f$ from $V$ to the space of Radon measures on $[0, 1) \times \R$ such that 
\begin{itemize}
\item $\mu_f$ is invariant by $T_f$;
\item $\mu_f \neq c\LebIR,$ for any $c \in \R \setminus \{0\}$;
\item $\mu_f \not\perp \LebIR$.
\end{itemize}
\end{proposition}

Given $T$ and $f$ as above we will construct the measure $\mu_f$ by using the ergodic decomposition of the measure $\textup{Leb}_{[0, 1)\times[-1/2,1/2]}$ with respect to the first return map of $T_f$ to the set {$[0, 1) \times [-1/2,1/2]$}. Let us denote the first return map by $\tilde T_f$. Recall that it follows from a result by Atkinson \cite{atkinson_recurrence_1976} that if $T$ is ergodic, then almost every point is \emph{recurrent}, that is, for every $\epsilon>0$, there exists arbitrarily large $N>0$ such that $|S_nf(x)|<\epsilon$. In particular, the map $\tilde T_f:[0, 1)\times[-1/2,1/2]\to[0, 1)\times[-1/2,1/2]$ is well-defined (for a.e. point) and preserves $\textup{Leb}_{[0, 1)\times[-1/2,1/2]}$. 

	For any $f\in V$, let us consider the ergodic decomposition of $\textup{Leb}_{[0, 1)\times[-1/2,1/2]}$ with respect to $\tilde T_{f}$:
	\begin{equation}\label{eq: erg_decomp}
\textup{Leb}_{[0, 1)\times[-1/2,1/2]}=\int_{[0, 1)\times[-1/2,1/2]/ \textup{Inv}(\tilde T_f)}p_{\bar x}\,d\rho_f(\bar x),
	\end{equation}
	where $\rho_f$ is a probability measure on the space of ergodic components $[-1/2,1/2]\times\R/ \textup{Inv}(\tilde T_f)$ and, for $\rho_f$-a.e. $\bar x \in [0, 1)\times[-1/2,1/2]/ \textup{Inv}(\tilde T_f)$, $p_{\bar x}$ is an ergodic measure with respect to $\tilde T_f$ (see \cite{Rohlin_general_1948}).

	If $\tilde T_f$ is ergodic, then the space of invariant components is trivial, and hence, it is one point, while if it is not ergodic, then the space is expected to be uncountable. Thus, to get a decomposition of $\textup{Leb}_{[0, 1)\times[-1/2,1/2]}$ that will allow us to assign a measure to a point and a cocycle measurably, we need to consider a different decomposition. 
	
	Since $(\tilde T_f,p_{\bar x})$ is ergodic for $\rho_f$-a.e. $\bar x\in [-1/2,1/2]\times\R/ \textup{Inv}(\tilde T_f)$, in view of Birkhoff Ergodic Theorem, $p_{\bar x}$-almost every point is generic, that is, $\rho_{\bar x}$-a.e. $x\in[-1/2,1/2]\times\R$ satisfies
	\[
	\lim_{n\to\infty}\frac{1}{n}\sum_{i=0}^{n-1}\delta_{\tilde T_f ^i x}=p_{\bar x},
	\]
where the limit is taken in the weak-* topology.
	In particular, almost every point with respect to $\textup{Leb}_{[0, 1)\times[-1/2,1/2]}$ is generic for a measure $p_{\bar x}$, for a unique $\bar x\in [0, 1)\times[-1/2,1/2]/ \textup{Inv}(\tilde T_f)$. Let us denote this measure by $p_{x,f}$. We have the following result.
	\begin{lemma}\label{lem: decomp}
		We have the following decomposition:
		\[
		\textup{Leb}_{[0, 1)\times[-1/2,1/2]}=\int_{[0, 1)\times[-1/2,1/2]}p_{x,f}\, dx.
		\]
	\end{lemma}
\begin{proof}
	By \eqref{eq: erg_decomp} we have
	\[
	\begin{split}
	\textup{Leb}_{[0, 1)\times[-1/2,1/2]}&=\int_{[0, 1)\times[-1/2,1/2]/ \textup{Inv}(\tilde T_f)}p_{\bar x}\,d\rho_f(\bar x)\\&
	=\int_{[0, 1)\times[-1/2,1/2]/ \textup{Inv}(\tilde T_f)}\int_{[0, 1)\times[-1/2,1/2]}p_{\bar x}\,dp_{\bar x}(x)\,d\rho_f(\bar x)\\
	&=\int_{[0, 1)\times[-1/2,1/2]/ \textup{Inv}(\tilde T_f)}\int_{[x|p_{x,f}=p_{\bar x}]}p_{\bar x}\,dp_{\bar x}(x)\,d\rho_f(\bar x)
	\end{split}
	\]
	Since the set $[x|p_{x,f}=p_{\bar x}]$ coincides with the invariant 
	component $\bar x$ (up to the measure $p_{\bar x}$), we get
	\[
	\begin{split}
	\textup{Leb}_{[0, 1)\times[-1/2,1/2]}&=\int_{[0, 1)\times[-1/2,1/2]/ \textup{Inv}(\tilde T_f)}\int_{\bar x}p_{x,f}\,dp_{\bar x}(x)\,d\rho_f(\bar x)\\
	&=\int_{[0, 1)\times[-1/2,1/2]}p_{x,f}\,d\textup{Leb}(x).
	\end{split}
	\]
\end{proof}
\begin{remark}
	By the Birkhoff Ergodic Theorem, the decomposition in the above lemma is 
	trivial (i.e., every measure in the decomposition is Lebesgue) if and only if $\tilde T_f$ is ergodic, which, by recurrence, is 
	equivalent to the ergodicity of $T_f$.
\end{remark}
We need the following simple lemma for the properties of the set of point convergence.
\begin{lemma}\label{eq: convergentmes}
	Let $(f_n)_{n\in\N}$ be a sequence of continuous functions on a topological space $X$ to a complete metric space $(Y,d)$. Then the set 
	\[
	C:=\left\{x\in X\,\left|\, \lim_{n\to\infty} f_n(x)\text{ exists}\right\}\right.
	\]
	is Borel.
\end{lemma}
\begin{proof}
	Since the space $(Y,d)$ is complete, we have
	\[
	C=\left\{x\in X\,\left|\, (f_n(x))_{n\in\N}\text{ is a Cauchy sequence}\right\}\right..
	\]
	Then 
	\[
	C=\bigcap_{m=1}^{\infty}\bigcup_{N=1}^{\infty}\bigcap_{k=N}^\infty\bigcap_{\ell=N}^\infty\{x\in X\,|\, |f_k(x)-f_l(x)|<1/m \},
	\]
	which, together with the assumption of continuity, finishes the proof. 
	
\end{proof}
We now turn into one of the main measurability results, which is going to be used to construct proper measure assignments. First, we introduce some auxiliary sets. Take 
\[
D:=\{(x,f)\in \left([0, 1)\times[-1/2,1/2]\right)\times W\ |\ S_nf(x)\neq \pm 
1/2, \text{ for any }n\in\Z\}
\]
and
\[
\begin{split}
	E&:=\left\{(x,f) \in\left([0, 1)\times[-1/2,1/2]\right)\times W\ \left|\ \begin{array}{l} T_f^k(x)\text{ 
		does not belong to the vertical line 
		given} \\ \text{by the discontinuity of }f\text{ or }T, \text{ for any }k\in\Z \end{array} \right\}\right.
\end{split}.
\]
Note that the complements of these sets are of zero measure. Indeed, the points that do not belong to one of those sets satisfy one of countably many linear equations. 

 In the following, we denote by $\mathcal M$ the space of probability measures on $[0, 1)\times[-1/2,1/2]$. 

	\begin{proposition}\label{prop: msrble}
		The assignment $(x,f)\mapsto p_{x,f}$, as in Lemma \ref{lem: decomp}, is measurable with respect 
		to the product topology in the domain and the weak-* topology in the 
		image.
	\end{proposition}
\begin{proof}
Consider the function $G_n:\left([0, 1)\times[-1/2,1/2]\right)\times W\to 
\mathcal 
M$ given by the formula
\[
G_n(x,f):=\frac{1}{n}\sum_{i=0}^{n-1}\delta_{\tilde T_f x}.
\]
On $\left([0, 1)\times[-1/2,1/2]\right)\times W$ we consider the product of the
Lebesgue measures. We will show that $G_n$ is continuous on a full-measure set.


Fix $(x,f)\in D\cap E$ and let $(x_m,f_m)_{m \in \N}$ be a sequence in $D \cap E$ converging to $(x,f)$. Let $p_1,\ldots, p_{m}$ denote the discontinuities of $f$ and $s_1,\ldots, s_{d - 1}$ the discontinuities of $T$. We will prove that 
\begin{equation}\label{eq: contofmeas}
	\lim_{m\to\infty} G_n(x_m,f_m)=G_n(x,f).
	\end{equation}
Let $\epsilon>0$ such that
\begin{equation}\label{eq:defepsilon}
	\begin{split}
\epsilon<\frac{1}{2}\min\Bigg\{ & 
\min_{0 \leq k \leq N(x)}(|S_{k}f(x)-1/2|),\, 
\min_{0 \leq k \leq N(x)}(|S_{k}f(x)+1/2|),\\&\min_{\substack{0 \leq k \leq N(x)\\ 
	1 \leq \ell \leq m}} |T^k(x)-p_\ell|, \min_{\substack{0 \leq k \leq N(x) \\
	1 \leq j < d}} |T^k(x)-s_j|\Bigg\},
\end{split}
\end{equation}
where $N(x)$ is $(n-1)$-th return time of $x$ to $[0, 1)\times[-1/2,1/2]$ via 
$T_f$. 
Note that by definition of $D$ and $E$, such $\epsilon$ exists.
Take $M\in \N$ such that for every $m\ge M$ we have 
\[
\max_{0 \leq k \leq N(x)}\left\{|S_{k}f(x)-S_{k}f_m(x_m)|, |T^k(x)-T^k(x_m)| 
\right\}<\epsilon.
\]
Such $M$ exists due to \eqref{eq:defepsilon}.
In particular, the return times up to $N(x)$ of iterations of $T_f$ to 
$[0, 1)\times[-1/2,1/2]$ are identical as those of $x$. Hence, the 
distance of $\delta_{\tilde T_f ^i x}$ and $\delta_{\tilde T_f^i x_m}$ in 
Levy-Prokhorov metric is less than $\epsilon$ for every $m>M$. This finishes 
the proof of continuity of $G_n$ on $D\cap E$. 

It follows, by Lemma \ref{eq: convergentmes}, that the set $\Omega$ of points $(x,f)\in D\cap E$ such that $x$ is generic for some $T_f$-invariant measure is measurable. Since for any $f\in W$, the set of points $x$ which are generic for some $T_f$-invariant measure is of full measure, the set $\Omega$ is of full product measure

To conclude the proof of the proposition, it suffices to notice that $\lim_{n\to\infty} 
G_n(x,f)= p_{x,f}$ for every $(x,f)\in\Omega$. In other words, $(x,f)\mapsto p_{x,f}$ on $\Omega$ is a point-wise limit of continuous maps. Thus, it is well-defined and measurable on a full measure set.
\end{proof} 

The decomposition obtained in the above proposition is not done on the whole space but rather on a subset of $[0, 1)\times\R$. We now show how to obtain a measure on $[0,1) \times \R$, which is absolutely continuous with respect to the Lebesgue measure, using the decomposition above.

For any $f\in W$ and any $N \in \N$, let $T_{f, N}:[0, 1)\times[-N,N] \to [0, 1)\times[-N,N] $ be the first return map of $T_f$ to $[0, 1)\times [-N,N]$. Note that the map $\tilde T_f$ can be viewed as the first return map of $T_{f, N}$ to $[0, 1)\times[-1/2,1/2]$. Denote by $\mathcal M^N$ the set of finite Radon measures on $[0, 1)\times [-N,N]$. 

The following is a simple corollary of Proposition \ref{prop: msrble}. 
\begin{lemma}\label{lem: Nextension}
 There exists a measurable assignment 
 \[
 ([0, 1)\times[-N,N])\times W\ni (x,f)\mapsto p^N_{x,f}\in\mathcal M^N,
 \]
 where 
 \[
 \mu^f_N:=\int_{[0, 1)\times[-N,N]} p^N_{x,f}\, dx<\!\!<\textup{Leb}_{[0, 1)\times [-N,N]}
 \]
and the integrated measures are invariant and ergodic. 
Moreover, for every $N\ge 1$ we have
 \[
 p^N_{x,f}|_{[0, 1)\times [-1/2,1/2]}=p_{x,f}
 \]
 and if $N_1<N_2$, then
 \[
 p^{N_2}_{x,f}|_{[0, 1)\times [-N_1,N_1]}=p^{N_1}_{x,f}.
 \]
\end{lemma}
\begin{proof} 
Let $N\in\N$. For any $m\in\N$ denote
\[
U_m:=\left\{x\in[0, 1)\times[-1/2,1/2]\,\left|\, \text{$m$ is the first return time of $x$ via $T_{f, N}$ to $[0, 1)\times[-1/2,1/2]$}\right\}\right..
\]
Note that the first return map to $[0, 1)\times[-1/2,1/2]$ via $T_{f, N}$ is equal $\tilde T_f$. Consider the measure on $[0, 1)\times[-1/2,1/2]$ given by
\[
\mu^N_f:=\sum_{m=1}^\infty \sum_{j=0}^{m-1} (T_{f, N})_*\left(\textup{Leb}_{[0, 1)\times[-1/2,1/2]}|_{U_m}\right).
\]
Since each of the summands is absolutely continuous w.r.t. Lebesgue measure, then so is $\mu^N_f$.
By Lemma \ref{lem: decomp} we have
\[
\mu^N_f=\int_{[0, 1)\times[-1/2,1/2]}\sum_{m=1}^\infty \sum_{j=0}^{m-1} (T_{f, N})^j_*p_{x,f}|_{U_m}\,dx
\]
Take 
\[
p^N_{x,f}:=\sum_{m=1}^\infty \sum_{j=0}^{m-1} (T_{f, N})^j_*p_{x,f}|_{U_m}.
\]
Since the restriction of measures and taking images of measures are measurable operations, by Proposition \ref{prop: msrble}, we obtain the measurability of the assignment $(x,f)\mapsto p^N_{x,f}$. 
The invariance and ergodicity of $p^N_{x,f}$ follow directly from the invariance and ergodicity of $p_{x,f}$, while the last two equations follow from the construction.
\end{proof}
\begin{remark}\label{rem: notnecLeb}
 With the notation of the above proof, we actually have that $\mu_f^N$ is a restriction of the Lebesgue measure on the strip to the disjoint union $\bigsqcup_{m=1}^{\infty}\bigsqcup_{j=0}^{m-1}(T_{f, N})^j(U_m)$. This does \textbf{not} have to be the whole measure $Leb_{[0, 1)\times[-N,N]}$.
\end{remark}

In the following result, we pass from the decomposition on the bounded subsets to the decomposition on the whole strip. We will denote by $\mathcal M_{\infty}$ the space of Radon measures on $[0, 1)\times \R$.
\begin{proposition}
\label{prop: msbdecomp}
	There exists a measurable assignment 
	\[
	\left([0, 1)\times[-1/2,1/2]\right)\times W\ni (x,f)\mapsto \tilde p_{x,f}\in \mathcal M_{\infty},
	\]
 such that
	\begin{itemize}
		\item for a.e every $(x,f)\in \left([0, 1)\times[-1/2,1/2]\right)\times W$ with respect to the product Lebesgue measure, the measure $\tilde p_{x,f}$ is $T_f$-invariant and $\tilde p_{x,f}|_{[0, 1)\times[-1/2,1/2]}=p_{x,f}$,
		\item $\tilde\mu_f:=\int_{[0, 1)\times[-1/2,1/2]}\tilde p_{x,f}\, dx$ is well defined and $\tilde\mu_f<\!\!<\textup{Leb}_{[0, 1)\times \R}$.
	\end{itemize}
\end{proposition}
\begin{proof}
 For every $x\in [0, 1)\times[-1/2,1/2]$ and $f\in W$ consider the sequence of measures $(p_{x,f}^N)_{N \in \N}$ given by the Lemma \ref{lem: Nextension}, here seen as measures on $[0, 1)\times\R$. 
 Define the measure $\tilde p_{x,f}$ by putting for every compact subset $K\subseteq [0, 1)\times\R$
 \[
\tilde p_{x,f}(K)= p_{x,f}^{N_K}(K),\text{ where }N_K:=\min\{N\in \N\,|\,K\subseteq [0, 1)\times [-N,N]\}.
 \]
 Note that in view of Lemma \ref{lem: Nextension}, we can replace in the above definition $N_K$ by any $N$ bigger than $N_K$. Hence, the above equation really defines a measure.
 
 We claim that 
 \begin{equation}\label{eq: partconverg}
 \lim_{N\to\infty} p_{x,f}^N= \tilde p_{x,f}
 \end{equation}
 exists. Indeed, if $g$ is a continuous function on $[0, 1)\times\R$ with a compact support $K_g$, then there exists $N_g\in\N$, such that $K_g\subseteq [0, 1)\times[-N_g,N_g]$. Then, by Lemma \ref{lem: Nextension} we get 
\[
\lim_{N\to\infty}\int_{[0, 1)\times \R}g\,dp_{x,f}^N=\int_{[0, 1)\times [-N_g,N_g]}g\,dp_{x,f}^{N_g}=
\int_{[0, 1)\times [-N_g,N_g]}g\,d\tilde p_{x,f}=\int_{[0, 1)\times \R}g\,d\tilde p_{x,f}.
\]
{Note that the assignment $P(x,f):= \tilde p_{x,f}$ is measurable. Indeed, recall that the topology of $\mathcal M_{\infty}$ is generated by the sets of the form $C(\nu,f_1,\ldots,f_n,\epsilon_1,\ldots,\epsilon_n)=\{\rho\in\mathcal M_{\infty}\mid |\int f_i\,d\rho- \int f_i\,d\nu|<\epsilon_i\}$, where $n\in\N$, $\nu\in\mathcal M_{\infty}$, $\epsilon_1,\ldots,\epsilon_n>0$ and $f_1,\ldots,f_n\in C_c([0, 1)\times\R)$. Then, there exists $N\in\N$ such that the supports of functions $f_1,\ldots,f_n$ are included in $[0, 1)\times[-N,N]$. Thus
\[
\begin{split}
&\{(x,f)\in([0, 1)\times[-1/2,1/2])\times W\mid \left|\int f_i\,d\tilde p_{x,f}- \int f_i\,d\nu\right|<\epsilon_i\}=\\
&\{(x,f)\in([0, 1)\times[-1/2,1/2])\times W\mid \left|\int f_i\,d p_{x,f}^N- \int f_i\,d\nu\right|<\epsilon_i\}.
\end{split}
\]
Since the assignment $(x,f)\mapsto p_{x,f}^N$ is measurable, so is the above set. By the fact that all parameters $\nu$, $f_j$, and $\epsilon_j$ are arbitrary, it follows that $P$ is measurable.
}

In a similar fashion, if $A\subseteq [0, 1)\times\R $ is a bounded set, then since $f$ is bounded, there exists $N\in\N$ such that $A,T_f A\subseteq [0, 1)\times[-N,N]$. Since $p_{x,f}^N$ is $T_{f, N}$-invariant, then $\tilde p_{x,f}(A)=\tilde p_{x,f}(T_f(A))$. Hence, the measure $\tilde p_{x,f}$ is $T_f$ invariant.

 Finally, it remains to notice that by Lemma \ref{lem: Nextension}, we have for every $N\in\N$ that 
 \begin{equation}\label{eq: abscontLeb}
\mu_f^N=\int_{[0, 1)\times[-1/2,1/2]}p_{x,f}^N\, dx<\!\!<\int_{[0, 1)\times[-N,N]}p_{x,f}^N\, dx=\textup{Leb}_{[0, 1)\times[-N,N]}.
 \end{equation}
 Then by \eqref{eq: partconverg} we get that $\tilde\mu_f:=\lim_{N\to\infty} \mu_f^N$ is well defined, $\tilde\mu_f|_{[0, 1)\times[-N,N]}=\mu_f^N$ and by \eqref{eq: abscontLeb}, we have that $\tilde\mu_f<\!\!<\textup{Leb}_{[0, 1)\times \R}$.
\end{proof}
\begin{remark}
 It is worth mentioning that the measure $\tilde\mu_f$ obtained in the above proposition does not need to be the whole Lebesgue measure on the strip $[0, 1)\times\R$. Indeed if the function $f$ is a bounded coboundary then the support of $\tilde\mu_f$ is contained in the set $[0, 1)\times [-\|f\|_{\infty}-1,\|f\|_{\infty}+1]$ (see also Remark \ref{rem: notnecLeb}).
\end{remark}

 One can observe that until now, we have yet to use the assumption of the non-ergodicity of the transformations considered. Indeed, this assumption is only used in the following proof of the Proposition \ref{prop:measurable_dependence_measures}.

\begin{proof}[Proof of Proposition \ref{prop:measurable_dependence_measures}]
 Consider the assignment $(x,f)\mapsto \tilde p_{x,f}$ given by Proposition \ref{prop: msbdecomp} and denote it by $P:([0, 1)\times[-1/2,1/2])\times W\to \mathcal M_{\infty}$. Since it is a measurable map, by Lusin's Theorem, there exists a compact subset $K\subseteq ([0, 1)\times[-1/2,1/2])\times W$ of positive measure such that $P|_K$ is continuous. 
Let $(x_0,f_0)\in K$ be a point of density of $K$. Since $f_0\in W$, we have that $\tilde p_{x_0,f_0}|_{[0, 1)\times[-1/2,1/2]}\neq \textup{Leb}_{[0, 1)\times[-1/2,1/2]}$. In particular, if $d_{LP}$ denotes the Levy-Prokhorov metric on $\mathcal M$ (not $\mathcal M_{\infty}$!) there exists $\epsilon>0$ such that $\textup{Leb}_{[0, 1)\times[-1/2,1/2]}\notin B_{LP}(\tilde p_{x_0,f_0}|_{[0, 1)\times[-1/2,1/2]},\epsilon)$. 
 Since $P|_{K}$ is continuous, there exists $\delta>0$ such that for every $(x,f)\in K\cap D((x_0,f_0),\delta)=:\bar K$, we have $d_{LP}(\tilde p_{x_0,f_0}|_{[0, 1)\times[-1/2,1/2]}, \tilde p_{x,f}|_{[0, 1)\times[-1/2,1/2]})<\epsilon/2$. Since $(x_0,f_0)$ is a density point of $K$, we have $\textup{Leb}_{[0, 1)\times[-1/2,1/2]}(\bar K)>0$. 
 
 Let $V\subseteq W$ be the set of those elements $f$ for which sets $K_f:=\{x\in[0, 1)\times[-1/2,1/2]| (x,f)\in \bar K\}$ have positive measure. Since $\bar K$ is of positive measure, then so is $V$. Consider the map 
 \[
 V\ni f\mapsto \bar\mu_f:= \int_{K_f} \tilde p_{x,f} dx.
 \]
 Note that it is actually the integration of $P|_{\bar K}$, a continuous function, with respect to the $x$-coordinate. In particular, it is continuous. Let $\tilde\mu_f:=\frac{1}{Leb(K_f)}\bar\mu_f$. Note that $Leb(K_f)=\bar\mu_f([0, 1)\times[-1/2,1/2])$. Since the map $\mathcal M\ni \nu\mapsto \nu([0, 1)\times[-1/2,1/2])$ is measurable, then so is $V\ni f\mapsto \tilde\mu_f\in\mathcal M$.
 
 Moreover, since $\bar K\subseteq D((x_0,f_0),\delta)$, by the choice of $\delta$ we have that measure $\tilde\mu_f$ is a convex combination of measures from the set $B_{LP}(p_{x_0,f_0},\epsilon/2)$ and as such, $\mu_f\in B_{LP}(p_{x_0,f_0},\epsilon/2)$. In particular 
 \[
 \tilde\mu_f|_{[0, 1)\times[-1/2,1/2]}\neq \textup{Leb}_{[0, 1)\times[-1/2,1/2]},
 \]
 hence
 \[
 \tilde\mu_f\neq \textup{Leb}_{[0, 1)\times \R}.
 \]
 Moreover, 
 \[
 \tilde\mu_f=\frac{1}{Leb( K_f)}\int_{K_f} \tilde p_{x,f} \,dx<\!\!<\int_{[0, 1)\times[-1/2,1/2]} \tilde p_{x,f} \,dx=\tilde\mu_f<\!\!< \textup{Leb}_{[0, 1)\times\R},
 \]
 where $\hat \mu_f$ is given by Proposition \ref{prop: msbdecomp}.
 This finishes the proof.

\end{proof}

We will now show Lemma \ref{lem: contemb}, which in turn was used earlier to show that Theorem \ref{thm:corollary} follows from Theorem \ref{thm:main_thm}. First, we recall some classical facts from the general ergodic theory. 
Let $(X,\mathcal B,\nu)$ be a standard probability space. We denote by $\textup{Aut}(X)$ the space of all $\nu$-measure-preserving automorphisms on $X$. Let $(B_k)_{k \in \N}$ be a countable family of subsets generating $\mathcal B$. It is a classical fact (see, e.g., \cite{aaronson_introduction_1997}), that this space endowed with the metric
\begin{equation}\label{eq: defofAutmetric}
d_{\textup{Aut}(X)}(S_1,S_2)=\sum_{k \in \N}2^{-k}\left(\nu(S_1(B_k)\triangle S_2(B_k))+\nu(S_1^{-1}(B_k)\triangle S_2^{-1}(B_k))\right)
\end{equation}
is a Polish space. Let also $\textup{Erg}(X)\subseteq \textup{Aut}(X)$ be the subset of ergodic automorphisms. It is a classical fact that it is measurable (see, e.g., \cite{halmos_lectures_1960}).

We now show that by twitching both the IET and the cocycle, the first return maps to the cylinder in $[0, 1)\times \R$ obtained in this way are close in an appropriate space of automorphisms.
\begin{lemma}\label{lem: contemb}
 With the notation as in Lemma \ref{lem: Nextension}, the map 
 \[
 \left((S_0\times\Lambda^{\mathcal A}\right)\cap \textup{Erg}([0, 1)))\times C_{m,M}\ni (T,f)\mapsto \tilde T_{f, N}\in \textup{Aut}\left([0, 1)\times[-N,N],\mathcal B, \textup{Leb}_{[0, 1)\times [-N,N]}\right)
 \]
 is continuous for every $N\in\R_{>0}$.
\end{lemma}
\begin{proof}
We will prove the lemma for $N=1/2$; for other cases, the proof differs only by proper rescaling of measures. Let $(B_k)_{k \in \N}$ be a family of rectangles generating the Borel $\sigma$-algebra on $[0, 1)\times[-N,N]$. We consider the metric $d:=d_{\textup{Aut}([0, 1)\times[-N,N])}$ as in \eqref{eq: defofAutmetric}.

 Fix $\pi\in S_0$ and let $(T_n,f_n)_{n \in \N}$ be a sequence converging to $(T,f)$ in $\{\pi\}\times\Lambda^{\mathcal A}\times C_{m,M}$ with respect to the product metric (on $\Lambda^{\mathcal A}$ we 
 consider standard Euclidean metric), with $T_n$ and $T$ being ergodic. Denote by $T_{n,f_n}$ the skew product given by $(T_n,f_n)$. Note that due to the ergodicity assumption, the first return transformations under consideration are all well-defined due to recurrence. Let $\varepsilon>0$ and let $K$ be such that 
 \[\sum_{i=K+1}^{\infty}2^{-i+1}\textup{Leb}_{[0, 1)\times [-1/2,1/2]}(B_i)<\varepsilon/2.\]
 It is enough to show that there exists $L\in\N$ such that for every $n\ge L$, we have 
 \[
 \sum_{i=1}^{K}2^{-i}\left(\textup{Leb}_{[0, 1)\times [-1/2,1/2]}(\tilde T_{n,f_n}(B_i)\triangle \tilde T_f(B_i))+\textup{Leb}_{[0, 1)\times [-1/2,1/2]}(\tilde T_{n,f_n}^{-1}(B_i)\triangle \tilde T_f^{-1}(B_i)\right)<\varepsilon/2.
 \]
 On the other hand, to show the above inequality, it is enough to find $L_i$ for every $i\in\{1,\ldots,K\}$ such that for every $n\ge L_i$
 \[
 \textup{Leb}_{[0, 1)\times [-1/2,1/2]}\left(\tilde T_{n,f_n}(B_i)\triangle \tilde T_f(B_i)\right)+ \textup{Leb}_{[0, 1)\times [-1/2,1/2]}\left(\tilde T_{n,f_n}^{-1}(B_i)\triangle \tilde T_f^{-1}(B_i)\right)<\varepsilon/2K
 \]
 and take $L:=\max_{i=1,\ldots,K} L_i$.

{ Fix $i\in\{1,\ldots,K\}$ and let $B:=B_i\cap [0, 1)\times [-1/2+\varepsilon/16K,1/2-\varepsilon/16K]$. By choosing $L_i$ sufficiently large, we can define a subset $\hat B\subseteq B$ with $\textup{Leb}(B\setminus\hat B)< \varepsilon/16K$ consisting of finitely many rectangles such that for any $n\ge L_i$, the restriction of the first return maps to $[0, 1)\times [-1/2,1/2]$ of $T_{n,f_n}$ and $T_f$ (resp. $T_{n,f_n}^{-1}$ and $T_f^{-1}$) to each of these rectangles are continuous and their return times are identical. From this, increasing $L_i$ further if necessary, we may assume that 
 \[
 \textup{Leb}_{[0, 1)\times [-1/2,1/2]}\left(\tilde T_{n,f_n}(\hat B)\triangle \tilde T_f(\hat B)\right)+\textup{Leb}_{[0, 1)\times [-1/2,1/2]}\left(\tilde T_{n,f_n}^{-1}(\hat B)\triangle \tilde T_f^{-1}(\hat B)\right)<\varepsilon/8K.
 \]
 Hence
 \[
 \begin{split}
 &\textup{Leb}_{[0, 1)\times [-1/2,1/2]}\left(\tilde T_{n,f_n}(B_i)\triangle \tilde T_f(B_i)\right)+\textup{Leb}_{[0, 1)\times [-1/2,1/2]}\left(\tilde T_{n,f_n}^{-1}(B_i)\triangle \tilde T_f^{-1}(B_i)\right)\\
 &\le\ \textup{Leb}_{[0, 1)\times [-1/2,1/2]}\left(\tilde T_{n,f_n}(B)\triangle \tilde T_f(B)\right)+\textup{Leb}_{[0, 1)\times [-1/2,1/2]}\left(\tilde T_{n,f_n}^{-1}(B)\triangle \tilde T_f^{-1}(B)\right)+\varepsilon/4K\\
 &\le\ \left(\tilde T_{n,f_n}(\hat B)\triangle \tilde T_f(\hat B)\right)+\textup{Leb}_{[0, 1)\times [-1/2,1/2]}\left(\tilde T_{n,f_n}^{-1}(\hat B)\triangle \tilde T_f^{-1}(\hat B)\right)+3\varepsilon/8K\le \varepsilon/2K,
 \end{split}
 \]
 which finishes the proof.}
\end{proof}

\bibliographystyle{acm}
\bibliography{Bibliography.bib}
\end{document}